\documentclass{amsart}
\usepackage{amsmath}
\usepackage{amssymb}
\usepackage{amsthm}
\usepackage{amscd}
\newcounter{pcounter}

\newcommand{\R}{\mathcal{R}}

\newcommand{\F}{\mathcal{F}}

\newcommand{\B}{\mathcal{B}}
\newcommand{\HS}{\mathcal{H}}
\newcommand{\K}{\mathcal{K}}
\newcommand{\EE}{\Bbb E}
\newcommand{\PP}{\Bbb P}

\newcommand{\ZZ}{\Bbb Z}

\newcommand{\NN}{\Bbb N}
\newcommand{\QQ}{\Bbb Q}
\newcommand{\CC}{\Bbb C}
\newcommand{\ip}[1]{\langle #1 \rangle}

\newcommand{\widetidle}{\widetilde}
\newcommand{\varespilon}{\varepsilon}

\newcommand{\actson}{\curvearrowright}

\newtheorem{theorem}{Theorem}
\newtheorem{definition}[theorem]{Definition}
\newtheorem{proposition}[theorem]{Proposition}
\newtheorem{cor}[theorem]{Corollary}
\newtheorem{lemma}[theorem]{Lemma}

\newcommand{\FF}{\Bbb F}

\DeclareMathOperator{\Cay}{Cay}

\DeclareMathOperator{\vol}{vol}

\DeclareMathOperator{\Span}{Span}

\DeclareMathOperator{\im}{im}
\DeclareMathOperator{\id}{Id}

\DeclareMathOperator{\tr}{tr}
\DeclareMathOperator{\supp}{supp}

\DeclareMathOperator{\Proj}{Proj}

\DeclareMathOperator{\Hom}{Hom}

\DeclareMathOperator{\Tr}{Tr}

\DeclareMathOperator{\opdim}{opdim}
\DeclareMathOperator{\wk}{wk}
\DeclareMathOperator{\Meas}{Meas}
\DeclareMathOperator{\dom}{dom}

\DeclareMathOperator{\ran}{ran}
\DeclareMathOperator{\graph}{graph}
\DeclareMathOperator{\Alg}{Alg}
\DeclareMathOperator{\Ball}{Ball}
\DeclareMathOperator{\dist}{dist}
\numberwithin{theorem}{section}

\begin{document}
\title[$l^{p}$-Dimension For Equivalence Relations]{An $l^{p}$-Version of Von Neumann Dimension for Representations of Equivalence Relations}      
\author{Ben Hayes}
\address{UCLA Math Sciences Building\\
         Los Angeles,CA 90095-1555} 
\email{brh6@ucla.edu}       
\date{\today}          
\maketitle

\begin{abstract} Following the methods of \cite{Me}, we introduce an extended version of von Neumann dimension for representations of a discrete, measure-preserving, sofic equivalence relation. Similar to \cite{Me}, this dimension is decreasing under equivariant maps with dense image, and in particular is an isomorphism invariant. We compute dimensions of $L^{p}(\R,\overline{\mu})^{\oplus n}$ for $1\leq p\leq 2.$ We also define an analogue of the first $l^{2}$-Betti number for $l^{p}$-cohomology of equivalence relations, provided the equivalence relations satisfies a certain ``finite presentation" assumption. This analogue of $l^{2}$-Betti numbers may shed some light on the conjecture that cost (as defined by Levitt in \cite{L}) is one more than $l^{2}$-Betti number (as defined by Gaboriau in \cite{Gab2}) of equivalence relations.
\end{abstract}
\tableofcontents

\section{Introduction}

First let us introduce some preliminary definitions.

\begin{definition} \emph{ A} discrete measure preserving equivalence relation is a triple \emph{ $(\R,X,\mu)$ where $(X,\mu)$ is a standard probability space. And $\R\subseteq X\times X,$ is a  subset satisfying the following properties:}\end{definition}

\begin{list}{Property \arabic{pcounter}:~}{\usecounter{pcounter}}
\item for almost ever $x\in X,$ we have $(x,x)\in \R:$
\item for almost every $x\in X,$ and for every $y,z\in X$ such that $(x,y),(y,z)\in \R$ we have $(x,z)\in \R,$
\item for almost every $x\in X,$ and for every $y\in X$ such that $(x,y)\in \R$ we have $(y,x)\in \R,$
\item for almost every $x\in X,$ $\{y\in X:(x,y)\in \R\}$ is countable,
\item for every $B\subseteq \R$ so that  $x\mapsto |\{y:(x,y)\in B\}|$ is measurable we have $y\mapsto |\{x:(x,y)\in B\}|$ is measurable, we shall call such sets \emph{measurable} subsets of $\R,$
\item for every $B\subseteq \R$ measurable we have
\[\int_{X}|\{y:(x,y)\in B\}|,d\mu(x)=\int_{X}|\{y:(y,x)\in B\}|\,d\mu(x),\]

\end{list}

	We define the measure $\overline{\mu}$ on $\R$ by
\[\overline{\mu}(B)=\int_{X}|\{y:(x,y)\in B\}|\,d\mu(x).\]

\begin{definition} \emph{ Let $(\R,X,\mu)$ be a discrete measure preserving equivalence relation. A} partial $\R$-morphism \emph{ is a bimeasurable bijection $\phi\colon A\to B$ where $A,B$ are measurable subsets of $X,$ and such that $(\phi(x),x)\in \R$ for almost every $x.$ We will denote $\phi^{-1}$ to be the partial morphism $\psi\colon B\to A$ which is the inverse of $\phi.$ It follows from our definitions that $\phi$ is necessarily measure-preserving.  We will set $A=\dom(\phi),B=\ran(\phi).$ If $C\subseteq X$ is measurable, we let $\id_{C}\colon C\to C$ be the partial morphism which is the identity on $C.$  We will let $[[\R]]$ denote the set of all partial $\R$-morphisms, we let $[\R]=\{\psi\in [[\R]]:\mu(\dom(\phi))=1\},$ we will identify elements in $[[\R]]$ when the set on which they differ is null. If $A\subseteq X$ is measurable and $\mu(A)>0$, we let $\R_{A}$ be the equivalence relation over $(A,\mu_{A})$ given by $\R_{A}=\R\cap A\times A,$ and $\mu_{A}=\frac{\mu\big|_{A}}{\mu(A)}$}\end{definition}

\begin{definition}\emph{ Let $\R$ be a discrete measure preserving equivalence relation on $(X,\mu).$ The} von Neumann algebra of $\R$ denoted $L(\R),$ \emph{is defined to be von Neumann algebra inside $B(L^{2}(\R,\overline{\mu}))$ generated by the operators $v_{\phi},\phi\in [[\R]]$ defined by}\end{definition}
\[v_{\phi}f(x,y)=\chi_{\dom(\phi^{-1})}(x)f(\phi^{-1}x,y).\]
That is, $L(\R)$ is defined to be the weak operator topology closure of 
\[\left\{\sum_{\phi \in [[\R]]}c_{\phi}v_{\phi}:c_{\phi}=0\mbox{ for all but finitely many $\phi$}\right\}.\]
We let $\tau\colon L(\R)\to \CC,$ be defined by
\[\tau(x)=\ip{x\chi_{D},\chi_{D}},\]
where $D=\{(x,x):x\in \R\}.$ It is known that 
\[\tau(xy)=\tau(yx),\]
\[\tau(x^{*}x)\geq 0,\mbox{with equality if and only if $x=0.$}\]

	For $x\in L(\R),$ we let $\|x\|_{2}=\tau(x^{*}x)^{1/2}.$

	If $M$ is a von Neumann algebra with a faithful normal tracial state $\tau,$ we shall use $|x|=(x^{*}x)^{1/2}$ and $\|x\|_{p}=\tau(|x|^{p})$ for $1\leq p<\infty,$ and $\|x\|_{\infty}$ for the operator norm of $x.$ In particular this applies to $M=M_{n}(\CC)$ and the trace given by $\tr=\frac{1}{n}\Tr,$ where $\Tr$ is the usual trace.

\begin{definition} \emph{ An} action \emph{of $\R$ on a Banach space V consists of map $[[\R]]\times V\to V,$ such that when we denote $\phi v=(\phi,v)$ the following axioms are satisfied}\end{definition}
\begin{list}{Axiom \arabic{pcounter}:~}{\usecounter{pcounter}}
\item $\phi(\psi v)=(\phi \psi)v$
\item $\phi^{-1}\phi v=\id_{\dom(\phi)}v,$
\item $\phi \cdot=\psi \cdot$ if $\phi=\psi$ almost everywhere,
\item $\id_{X}v=v.$
\item $\phi_{n}v\to \phi v,$ \mbox{ if $\phi_{n}\in [[\R]],$ and $\|\phi-\phi_{n}\|_{2}\to 0.$}
\end{list}
We also call $V$ a representation of $\R.$ We say that the action is \emph{uniformly bounded} if there is a constant $C>0$ so that
\[\|\phi \cdot v\|\leq C\|v\|\]
We say that two representations $V,W$  are \emph{isomorphic} if there is a bounded linear bijection $T\colon V\to W$ such that $T(\phi v)=\phi T(v)$ for all $\phi\in [[\R]],v\in V.$ 

	We say that two representations are \emph{weakly isomorphic} if for every $\varepsilon>0,$ there is an $\R$-invariant set $A\subseteq X$ of measure at least $1-\varepsilon,$ and a bounded linear bijection $T\colon \id_{A}(V)\to \id_{A}(W),$ such that for every $\phi\in [[\R]]$ with $\dom(\phi)\subseteq A,$ and for every $v\in \id_{A}V,$
\[T(\phi v)=\phi T(v).\]

	Here is a natural example arising from groups. Let $\Gamma$ be a countable discrete group and $\Gamma \actson (X,\mu)$ a free action. Let $\R$ be the corresponding equivalence relation. Consider the Zimmer cocycle $\theta\colon \R\to \Gamma \times X$ given by $\theta(x,y)\cdot y=x.$ Let $\Gamma$ have a uniformly bounded representation $\pi\colon \Gamma\to B(V),$ with $V$ a Banach space. Define a representation of $\R$ on $L^{p}(X,\mu,V)$  for $1\leq p<\infty,$ by
\[(\phi \cdot f)(x)=\chi_{\ran(\phi)}(x)\pi(\theta(x,\phi^{-1}x))f(\phi^{-1}x).\]

	For example, $\R$ acts naturally on $L^{2}(\R,\overline{\mu})$ by viewing $\R$ inside $L(\R).$ Let $\mathcal{H}\subseteq l^{2}(\NN, L^{2}(\R,\overline{\mu}))$ be $\R$-invariant, and $P_{\mathcal{H}}$ the projection onto $\mathcal{H}.$ Define
\[\dim_{L(\R)}(\mathcal{H})=\sum_{n=1}^{\infty}\ip{P_{\mathcal{H}}(\chi_{\Delta}\otimes e_{n}),\chi_{\Delta}\otimes e_{n}},\]
where $\chi_{\Delta}\otimes e_{n}(m)=\delta_{n=m}\chi_{\Delta}.$ 

	Then $\dim_{L(\R)}(\mathcal{H})$ obeys the usual properties of dimension:

\begin{list}{Property \arabic{pcounter}:~}{\usecounter{pcounter}}
\item$\dim_{L(\R)}(\mathcal{H})=\dim_{L(\R)}(\mathcal{K}),$ if there is a $\R$-equivariant bounded linear bijection from $H$ to $K$\\
\item $\dim_{L(\Gamma)}(\HS\oplus \K)=\dim_{L(\Gamma)}(\HS)+\dim_{L(\Gamma)}(\K).$\\
\item $\dim_{L(\Gamma)}(\HS)=0$ if and only if $\HS=0,$\\
\item $\dim_{L(\Gamma)}\left(\bigcap_{n=1}^{\infty}\HS_{n}\right)=\lim_{n\to \infty}\dim_{L(\Gamma)}(\HS_{n}),\mbox{ if $\HS_{n+1}\subseteq \HS_{n}$}$,\\
\item $\dim_{L(\Gamma)}\overline{\left(\bigcup_{n=1}^{\infty}\HS_{n}\right)}=\lim_{n\to \infty}\dim_{L(\Gamma)}(\HS_{n})\mbox{ if $\HS_{n}\subseteq \HS_{n+1}.$}$
\end{list}

	We extend this notion of dimension for a certain class of equivalence relations, called \emph{sofic} equivalence relations. The ideas are the same as in \cite{Me}. The basic idea is that ``dimension is entropy." For example, if $\R$ comes from a free, measure-preserving action of a countable discrete group $\Gamma \actson (X,\mu),$ then $L^{2}(\R,\mu)$ can be viewed as the direct integral of $l^{2}(\Gamma)$ over $(X,\mu).$ Further this identification describes the action of $\R$ as coming from the left regular representation of $\Gamma$ on $l^{2}(\Gamma).$ Tautologically, $l^{2}(\Gamma)$ is a subset of $\CC^{\Gamma},$ and this left-translation action can be seen as a restricted Bernoulli action. So one should be able to follow the theory of entropy for Bernoulli actions on compact metric spaces or standard probability spaces. This idea was used by Voiculescu in \cite{Voi} to express von Neumann dimension by a formula analogous to topological entropy in the case of amenable groups. Further, Gornay in \cite{Gor} used the idea of von Neumann dimension as a type of entropy to generalize von Neumann dimension to the case of an amenable group acting by left-translation on $l^{p}(\Gamma)^{\oplus n}$ (it is in fact defined for $\Gamma$-invariant subspaces of $l^{p}(\Gamma)^{\oplus n}.$

	We will thus define a upper and lower notions of $l^{p}$-dimension for sofic equivalence relations, denoted $\dim_{\Sigma,l^{p}}(V,\R),\underline{\dim}_{\Sigma,l^{p}}(V,\R)$  (here $\Sigma$ is a sofic approximation, a notion to be defined later) satisfying the following properties.

\begin{list}{Property \arabic{pcounter}: ~ }{\usecounter{pcounter}}
\item $\dim_{\Sigma,l^{p}}(W,\R)\leq \dim_{\Sigma,l^{p}}(V,\R)$ if there is a $\R$-equivariant bounded map $W\to V$ with dense image and the same for $\underline{\dim},$
\item $\mu(A)\dim_{\Sigma,l^{p}}(\id_{A}V,\R_{A})=\dim_{\Sigma,l^{p}}(V,\R)$and the same for $\underline{\dim}$
\item $\dim_{\Sigma,l^{p}}(V,\R)\leq \dim_{\Sigma,l^{p}}(W,\R)+\dim_{\Sigma,l^{p}}(V/W,\R),$ if $W\subseteq V$ is a closed $\R$-invariant subspace.
\item $\underline{\dim}_{\Sigma,l^{p}}(V,\R)\leq \underline{\dim}_{\Sigma,l^{p}}(W,\R)+\dim_{\Sigma,l^{p}}(V/W,\R),$ if $W\subseteq V$ is a closed $\R$-invariant subspace.
\item $\underline{\dim}_{\Sigma,l^{p}}(V,\R)\leq \dim_{\Sigma,l^{p}}(W,\R)+\underline{\dim}_{\Sigma,l^{p}}(V/W,\R),$ if $W\subseteq V$ is a closed $\R$-invariant subspace.
\item $\underline{\dim}_{\Sigma,l^{2}}(H,\R)=\dim_{\Sigma,l^{2}}(H,\R)= \dim_{L(\Gamma)}H$  if $H\subseteq l^{2}(\NN,L^{2}(\R,\overline{\mu}))$ is a closed $\R$-invariant subspace.
\item $\underline{\dim}_{\Sigma,l^{p}}(L^{p}(\R,\overline{\mu})^{\oplus n},\R)=\dim_{\Sigma,l^{p}}(L^{p}(\R,\overline{\mu})^{\oplus n},\R)=n$ for $1\leq p\leq 2.$\\

\end{list} 

	Also, in Section \ref{S:cohom}, if $\R$ is a sofic equivalence relation with sofic approximation, which satisfies a certain ``finite presentation" assumption, we define a number $c^{(p)}_{1,\Sigma}(\R),$ which is an $l^{p}$-analogue of $\beta^{(2)}_{1}(\R)+1.$ This number has the property that  $c^{(p)}_{1,\Sigma}(\R)\leq c(\R),$ where $C(\R)$ is the cost of $\R.$ Further, $\mu(A)(c^{(p)}_{1,\Sigma_{A}}(\R_{A})-1)\geq c^{(p)}_{1,\Sigma}(\R)-1.$ This is if we could find an equivalence relation with vanishing $l^{2}$-cohomology, but so that $c^{(p)}_{1,\Sigma}(\R)>1,$ then we could disprove the conjecture that $\beta^{(2)}_{1}(\R)=c(\R)+1.$ If in addition we could prove that $c^{(p)}_{1,\Sigma}(\R)>1$ for all $\Sigma,$ then $\R$ would necessarily have trivial fundamental group.

\section{Definition of The Invariants}

	We start with the definition of a sofic approximation. For this, we let $[[\R_{n}]]$ be the equivalence relation on $\{1,\cdots,n\}$ defined by declaring all points to be equivalent. Then $[[\R_{n}]]$ acts on $\{1,\cdots,n\}$ as before so we may view $[[\R_{n}]]\subseteq B(l^{2}(n))\cong M_{n}(\CC).$

\begin{definition}\emph{ A} sofic approximation \emph{of a discrete measure preserving equivalence relation $(\R,X,\mu)$ is a sequence of maps $\sigma_{i}\colon *-\Alg([[R]],\Proj(L^{\infty}(X,\mu)))\to M_{d_{i}}(\CC)$ such that }
\emph{such that $\sigma_{i}([[R]])\subseteq [[\R_{d_{i}}]],\sigma_{i}(\Proj(L^{\infty}(X,\mu))\subseteq \Proj(l^{\infty}(d_{i})),$ and}
\[\|\sigma_{i}(xy)-\sigma_{i}(x)\sigma_{i}(y)\|_{2}\to 0,\mbox{ \emph{for all $x,y\in *-\Alg(\Proj(L^{\infty}(X,\mu)),[[\R]])$}},\]
\[\|\sigma_{i}(x^{*})-\sigma_{i}(x)^{*}\|_{2}\to 0,\mbox{ for all $x$}\]
\[\tr\circ \sigma_{i}(x)\to \tau(x)\mbox{ \emph{for all $x\in *-\Alg(\Proj(L^{\infty}(X,\mu)),[[\R]])$}},\]
\[\sup_{i}\|\sigma_{i}(x)\|_{\infty}<\infty\mbox{ \emph{for all} $x$}.\]
\[\|\sigma_{i}(x+y)-\sigma_{i}(x)-\sigma_{i}(y)\|_{2}\to 0\emph{ for all $x,y\in A$}\]
\[\|\sigma_{i}(\lambda x)-\lambda\sigma_{i}(x)\|_{2}\to 0\emph{ for all $x\in A,\lambda\in \CC$}\]

\end{definition} 

	See \cite{EL} for more about sofic equivalence relations.

	One can actually relax the condition that $\sigma_{i}$ is everywhere defined by using the following proposition.

\begin{proposition}\label{P:extend} Let $\Phi$ be a graphing of $\R,$ and let $\mathcal{P}\subseteq \Proj(L^{\infty}(X,\mu)$ be such that $W^{*}(\{v_{\phi}pv_{\phi}^{*}:p\in P,\phi\in \Phi\})=L^{\infty}(X,\mu).$ Let $A=*-\Alg(P,\Phi).$ Suppose that there exists $\phi_{i}\colon A\to M_{d_{i}}(\CC)$ such that 
\[\phi_{i}(P)\subseteq D_{d_{i}}\mbox{ for all $\phi\in \Phi\cup \Phi^{*}\cup \{\id\},p\in P$}\]
\[\phi_{i}(\Phi)\subseteq [[R_{d_{i}}]].\]
\[\|\phi_{i}(xy)-\phi_{i}(x)\phi_{i}(y)\|_{2}\to 0,\mbox{ for all $x,y\in A$}\]
\[\|\phi_{i}(x^{*})-\phi_{i}(x)^{*}\|_{2}\to 0,\mbox{ for all $x\in A$}\]
\[\tr\circ \phi_{i}(x)\to \tau(x)\mbox{ for all $x\in A$}\]
\[\sup_{i}\|\phi_{i}(x)\|_{\infty}<\infty\mbox{ for all $x\in A$.}\]
\[\|\phi_{i}(x+y)-\sigma_{i}(x)-\phi_{i}(y)\|_{2}\to 0,\mbox{for all $x,y\in A.$}\]
\[\|\phi_{i}(\lambda x)-\lambda \phi_{i}(x)\|_{2}\to 0,\mbox{ for all $x\in A,\lambda\in \CC.$}\]

	Then there exists a sofic approximation $\sigma_{i}\colon *-\Alg(\Proj(L^{\infty}(X,\mu)),[[\R]])\to M_{d_{i}}(\CC)$ such that 
\[\|\rho_{i}(x)-\phi_{i}(x)\|_{2}\to 0\]
for all $x\in A.$ 
\end{proposition}

\begin{proof}  It is easy to see that for all $p\in \Proj(L^{\infty}(X,\mu))\cap A,\phi\in [[\R]]\cap A$ there are $q_{p,i}\in \Proj(l^{\infty}(d_{i})),v_{\phi,i}\in [[\R]]\cap A$ such that 
\[\sigma_{i}(p)=q_{p,i}\]
\[\sigma_{i}(\phi)=v_{\phi,i},\]
\[\|q_{p,i}-\sigma_{i}(p)\|_{2}\to 0,\]
\[\|v_{\phi,i}-\sigma_{i}(\phi)\|_{2}\to 0,\]
and such that $q_{\id,i}=\id,v_{\id,i}=\id.$ 

	Define
\[\widetilde{\phi_{i}}(x)=\phi_{i}(x)\mbox{ if $x\in A\setminus (\Proj(L^{\infty}(X,\mu))\cap A)\cup ([[\R]]\cap A)$},\]
\[\widetidle{\phi_{i}}(p)=q_{p,i},\mbox{ for $p\in \Proj(L^{\infty}(X,\mu))\cap A$}\]
\[\widetilde{\phi_{i}}(\phi)=v_{\phi,i}.\]

	It is easy to see that $\widetidle{\phi_{i}}$ satisfies the same hypotheses as $\phi_{i}.$ Hence, replacing $\widetidle{\phi_{i}}$ with $\phi_{i}$ we may as well assume that $\phi_{i}(\Proj(L^{\infty}(X,\mu)\cap A)\subseteq \Proj(l^{\infty}(d_{i})),$ and 
\[\phi_{i}([[\R]]\cap A)\subseteq [[\R_{d_{i}}]].\]

	For a given $x\in L(\R)$ invoke the Kaplansky Density Theorem to find $x_{n}\in A$ such that 
\[\|x-x_{n}\|_{2}<2^{-n},\|x_{n}\|_{\infty}\leq \|x\|_{\infty}.\]

	If $x\in \Proj(L^{\infty}(X,\mu))$ we may force $x_{n}\in \Proj(L^{\infty}(X,\mu))\cap A$ if $x\in [[\R]],$ we may force $x_{n}\in [[\R]]\cap A.$ An ultraproduct argument proves that for each $n,i$ there are $y_{n,i}\in M_{d_{i}}(\CC)$ so that 
\[\|y_{n,i}\|_{\infty}\leq \|x_{n}\|_{\infty}\]
\[\|y_{n,i}-\phi_{i}(x_{n,i})\|_{2}\to 0,\mbox{ as $i\to \infty$}.\]

	Again we may force $y_{n,i}\in \Proj(l^{\infty}(d_{i}))$ if $x_{n,i}\in L^{\infty}(X,\mu),$ whereas if $y_{n,i}\in [[\R]]$ we may force $y_{n,i}\in [[\R_{d_{i}}]].$ Choose $i_{1}\leq i_{2}\leq i_{3}\leq \cdots$ such that \
\[\|\phi_{i}(x_{j})-y_{j,i}\|_{2}<2^{-n}\mbox{ if $i\leq i_{n},1\leq j\leq n$}\]
\[\|y_{j,i}-y_{j+1,i}\|_{2}<2\cdot 2^{-j}\mbox{ if $i\leq i_{n},1\leq j\leq n-1$}\]
and define
\[\sigma_{i}(x)=y_{n,i} \mbox{ if $i_{n}\leq i<i_{n+1}.$}\]

	Note that if $m\geq n,$ and $i_{m}\leq i<i_{m+1},$ then
\begin{align*}
\|\sigma_{i}(x)-\phi_{i}(x_{n})\|_{2}&=\|y_{m,i}-\phi_{i}(x_{n})\|\\
&\leq 2^{-m}+\sum_{j=n}^{m-1}\|y_{j,i}-\phi_{i}(x_{j})\|+\|y_{j,i}-y_{j+1,i}\|_{2}\\
&\leq 2^{-m}+2\cdot 2^{-n}
\end{align*}

	And since $\sup_{i}\|\sigma_{i}(x)\|_{\infty}<\infty,$ it then becomes a simple exercise to verify that $\sigma_{i}$ has the desired properties.

\end{proof}

	We now proceed to state the definition of our extended von Neumann dimension, again the ideas are parallel to \cite{Me}.

\begin{definition} \emph{ Let $V$ be a separable Banach space with a uniformly bounded action of $\R,$ and let $q\colon W\to V$ be a bounded linear surjective map where $Y$ has the bounded approximation property. Let $\Phi\subseteq L(\R).$ For $F\subseteq \Phi$ finite, we define $\mathcal{W}_{k}(F)=\{\phi_{1}\cdots \phi_{j}:1\leq j\leq k,\phi_{j}\in F\}.$  A} $q$-dynamical filtration\emph{ consists of a pair $\F=((b_{\phi,j})_{(j,\phi)\in \NN\times \mathcal{W}(\Phi)},(W_{F,k})_{F\subseteq \Phi\mbox{ finite}}$ where}\end{definition}
\[b_{\phi,j}\in W,\]
\[\sup_{(j,\phi)}\|b_{\phi,j}\|<\infty,\]
\[q(b_{\id,j})\mbox{ is dynamically generating},\]
\[q(b_{\phi,j})=pv_{\phi}q(b_{j,\id}),\]
\[W_{F,k}\subseteq W_{F',k'}\mbox{ if $F\subseteq F',k\leq k',$}\]
\[W_{F,k}=\Span\{b_{j,\phi}:1\leq j\leq k,\phi\in\mathcal{W}_{k}(F)\}+\ker(q)\cap W_{F,k},\]
\[\ker(q)=\overline{\bigcup_{F,k}W_{F,k}\cap \ker(q)}.\]

\begin{definition} \emph{For $C>0,$ a Banach space $W$ is said to have the} $C$-bounded approximation property \emph{if there is a net $\theta_{\alpha}\colon W\to W$ of finite rank maps such that $\|\theta_{\alpha}\|\leq C,$ and $\theta_{\alpha}\to \id$ in the strong operator topology. We say that $W$ has the} bounded approximation property \emph{ if it has the $C$-bounded approximation property for some $C>0.$}\end{definition}

\begin{definition} \emph{A quotient dimension tuple is a tuple $((X,\mu),\R,\Phi,V,W,q,\Sigma)$ where $(X,\mu)$ is a standard probability space, $\R$ is a discrete measure-preserving equivalence relation on $(X,\mu),$ $\Phi\subseteq L(\R)$ is of the form $\Phi=\Phi_{0}\cup \mathcal{P},$ where $\Phi_{0}\subseteq [[\R]]$ is a graphing, and $1\in \mathcal{P}\subseteq\Proj(L^{\infty}(X,\mu))$ has $W^{*}(\{v_{\phi}pv_{\phi}^{*}:\phi \in \Phi_{0},p\in \mathcal{P})=L^{\infty}(X,\mu),$  $V$is a uniformly bounded representation of $\R,$ $W$ is a separable Banach space with the bounded approximation property, $q\colon W\to V$ is a bounded linear surjective map, and $\Sigma=(\sigma_{i}\colon *-\Alg([[\R]],\Proj(L^{\infty}(X,\mu)))\to M_{d_{i}}(\CC))$ is a sofic approximation.}
\end{definition}

\begin{definition} \emph{Let $((X,\mu),\R,\Phi,V,W,q,\Sigma)$ be a quotient dimension tuple. Let $\F=((b_{j,\phi},W_{F,k}))$ be a $q$-dynamical filtration. For $F\subseteq \Phi$ finite, $m\in \NN,\delta>0$ we let $\Hom_{\R,l^{p}}(\F,F,m,\delta,\sigma_{i})$ consists of all linear maps $T\colon W\to l^{p}(d_{i})$  with $\|T\|\leq 1,$ and such that there is an $A\subseteq \{1,\cdots,d_{i}\}$ with $|A|\geq (1-\delta)d_{i}$ so that for all $1\leq j\leq m,$ for all $\phi_{1},\cdots,\phi_{m}\in F$ we have}
\[\|T(b_{\phi_{1}\cdots\phi_{k},j})-\sigma_{i}(\phi)\cdots \sigma_{i}(\phi_{k})T(b_{\id,j})\|_{l^{p}(A)}<\delta\]
\[\|T\big|_{\ker(q)\cap W_{F,m}}\|\leq \delta.\]
\end{definition}

	The intuition for the preceding definition is as follows. If $T\colon V\to l^{p}(d_{i})$ were an honest equivariant map, then by composing with the quotient map we find a map $S\colon W\to l^{p}(d_{i})$ such that 
\[S(b_{\phi_{1}\cdots\phi_{k},j})=\sigma_{i}(\phi)\cdots \sigma_{i}(\phi_{k})S(b_{\id,j}),\]
\[S\big|_{\ker(q)}=0,\]
for all $\phi_{1},\cdots,\phi_{k}\in [[\R]]$ and $j\in \NN.$ So $\Hom_{\R,l^{p}}(\cdots)$ may be thought of as a space of almost equivariant maps. In this case, we must cut down by the set $A,$ in order to pass from one graphing of $\R$ to another.

\begin{definition} \emph{Let $(\R,X,\mu)$ be a discrete measure-preserving equivalence relation with a uniformly bounded representation on a Banach space $V.$ A} dynamically generating sequence \emph{ is a bounded sequence $S=(v_{j})_{j=1}^{\infty}$ in $V$ such that $\overline{\Span\{\phi v_{j}:j\in \NN,\phi \in [[\R]]\}}=V.$ If $\Sigma$ is a sofic approximation of $\R,$ and $\Phi=\Phi_{0}\cup \mathcal{P}\subseteq [[\R]]$ with $\Phi_{0}$ a graphing and $\mathcal{P}$ a set of projections so that $W^{*}(\{v_{\phi}^{*}pv_{\phi}:p\in \mathcal{P}\})=L^{\infty}(X,\mu),$ then the tuple $((X,\mu),\R,\Phi,V,S,\Sigma)$ will be called a} dimension tuple.\end{definition}

\begin{definition} \emph{Let $V$ be a Banach space and $n\in \NN.$ Let $\rho$ be a pseudonorm on $B(V,l^{p}(n)),$ if $A,B\subseteq B(V,l^{p}(n)),$ for $\varepsilon,M>0,$ we say that $A$ is} $(\varepsilon,M)$-contained in $B$\emph{ if for every $T\in A,$ there is an $S\in B,$ with $\|S\|\leq M$ and $C\subseteq\{1,\cdots,n\}$ with $|C|\geq (1-\varepsilon)n,$ so that $\rho(m_{\chi_{C}}(T-S))<\varepsilon.$ Similarly, if $\rho$ is a pseudonorm on $l^{\infty}(\NN,l^{p}(n))$ and $A,B\subseteq l^{\infty}(\NN,l^{p}(n))$ we say that $A$ is $\varepsilon$-contained in $B$ if for every $f\in A$ there is a $g\in B$ and $C\subseteq\{1,\cdots,n\}$ with $|C|\geq (1-\varepsilon)n$ so that $\rho(\chi_{C}(f-g))<\varepsilon.$ We shall use $d_{\varepsilon}(A,\rho),$ (respectively $d_{\varepsilon,M}(A,\rho)$) for the smallest dimension of a linear subspace which $\varepsilon$-contains (respectively $(\varepsilon,M)$-contains) $A.$ }
\end{definition}

	Note the difference between $\varepsilon$-containment as stated here and in \cite{Me}, this difference is why we have difficulty proving any sort of relation between extended von Neumann dimension for groups and for equivalence relations.

\begin{definition} \emph{Let $((X,\mu),\R,\Phi,V,W,q,\Sigma)$ be a quotient dimension tuple, and $\F$ a $q$-dynamical filtration. For a sequence of pseduonorms $\rho=(\rho_{i})$ on $B(W,l^{p}(d_{i}))$ we define}
\[\opdim_{\Sigma,M,l^{p}}(\F,\Phi,F,m,\delta,\varepsilon,\Phi,\rho)=\limsup_{i\to \infty}\frac{1}{d_{i}}d_{\varepsilon,M}(\Hom_{\R,l^{p}}(\F,F,m,\delta,\sigma_{i})),\]
\[\opdim_{\Sigma,l^{p}}(\F,\varepsilon,\Phi,\rho)=\inf_{F\subseteq \Phi\mbox{ \emph{finite}},m\in\NN,\delta>0}\opdim_{\Sigma,M,l^{p}}(\F,F,m,\delta,\varepsilon,\Phi,\rho),\]
\[\opdim_{\Sigma,M,l^{p}}(\F,\Phi,\rho)=\sup_{\varepsilon>0}\opdim_{\Sigma,M,l^{p}}(\F,\Phi,\varepsilon,\rho).\]
\end{definition}
	We also define $\underline{\opdim}_{\Sigma,M,l^{p}}(\F,\Phi,\rho)$ in the same way except using a limit infimum instead of a limit supremum.
	For later use, we note that if $\rho$ is a norm on $l^{\infty}(\NN)$ and $\F$ is as above, we use $\rho_{\F,i}(T)=\rho(j\mapsto \|T(b_{\id,j})\|).$

\begin{definition} \emph{Let $((X,\mu),\R,\Phi,V,W,q,\Sigma)$ be a quotient dimension tuple, and $\F$ a $q$-dynamical filtration. Define $\alpha_{\F}\colon B(V,l^{p}(d_{i}))\to l^{\infty}(\NN,l^{p}(d_{i}))$ by $\alpha_{F}(T)(n)=T(b_{\id,n}).$  We define}
\[f.\dim_{\Sigma,l^m{p}}(\F,F,m,\delta,\varepsilon,\Phi,\rho)=\limsup_{i\to \infty}\frac{1}{d_{i}}d_{\varepsilon}(\alpha_{\F}(\Hom_{\R,l^{p}}(\F,F,m,\delta,\sigma_{i})),\rho_{p,d_{i}}),\]
\[f.\dim_{\Sigma,l^{p}}(\F,\varepsilon,\Phi,\rho)=\inf_{F\subseteq \Phi\mbox{\emph{ finite}},m\in\NN,\delta>0}\opdim_{\Sigma,M,l^{p}}(\F,F,m,\delta,\varepsilon,\rho),\]
\[f.\dim_{\Sigma,l^{p}}(\F,\Phi,\rho)=\sup_{\varepsilon>0}f.\dim_{\Sigma,l^{p}}(\F,\varepsilon,\Phi,\rho).\]
\end{definition}

\begin{definition} \emph{Let $((X,\mu),\R,\Phi,V,S,\Sigma)$ be a dimension tuple. Let $\rho$ be a norm on $l^{\infty}(\NN).$  Let $\rho_{p,d_{i}}$ be the norm on $l^{\infty}(\NN,l^{p}(d_{i}))$ given by $\rho_{p,d_{i}}(f)=\rho(\|f\|_{p}).$ Let $S=(v_{j})_{j=1}^{\infty},$ set $V_{F,m}=\Span\{\phi v_{j}:\phi \in (F\cup \id\cup F^{*})^{m},1\leq j\leq m\}.$ Let $\alpha_{S}\colon B(V_{F,m},l^{p}(d_{i}))\to l^{\infty}(\NN,l^{p}(d_{i}))$ be given by $\alpha_{S}(T)(j)=\chi_{\{l\leq m\}}(j)T(v_{j}).$ We define}
\[f.\dim_{\Sigma,l^m{p}}(S,F,m,\delta,\varepsilon,\Phi,\rho)=\limsup_{i\to \infty}\frac{1}{d_{i}}d_{\varepsilon}(\alpha_{S}(\Hom_{\R,l^{p}}(S,F,m,\delta,\sigma_{i})),\rho_{p,d_{i}}),\]
\[f.\dim_{\Sigma,l^{p}}(S,\varepsilon,\Phi,\rho)=\inf_{F\subseteq \Phi\mbox{\emph{ finite}},m\in\NN,\delta>0}\opdim_{\Sigma,M,l^{p}}(S,F,m,\delta,\varepsilon,\rho),\]
\[f.\dim_{\Sigma,l^{p}}(S,\Phi,\rho)=\sup_{\varepsilon>0}f.\dim_{\Sigma,l^{p}}(S,\varepsilon,\Phi,\rho).\]
\end{definition}

	We shall define $\underline{f.\dim}_{\Sigma,l^{p}}(S,\Phi,\rho)$ for the same thing, except replacing all the limit supremums with limit infimums.

\begin{definition} \emph{A} product norm \emph{ on $l^{\infty}(\NN)$ is a norm $\rho$ such that $\rho(f)\leq \rho(g)$ if $|f|\leq |g|,$ and such that $\rho$ induces the topology of pointwise convergence on $\{f:\|f\|_{\infty}\leq 1\}.$}\end{definition}

	A typical example is 
\[\rho(f)=\left(\sum_{j=1}^{\infty}\frac{1}{2^{j}}|f(j)|^{p}\right)^{1/p}\]
for $1\leq p<\infty.$

	We will show that 
\[f.\dim_{\Sigma,l^{p}}(S,\Phi,\rho)=f.\dim_{\Sigma,l^{p}}(S',\Phi',\rho')\]
if $S,S'$ are two dynamically generating sequences, $\Phi,\Phi'$ are two graphings  and $\rho,\rho'$ are two product norms. Thus we can define $\dim_{\Sigma,l^{p}}(V,\R)$ to be either of these common numbers. Here is a rough outline of the proof. We first show that 
\[\opdim_{\Sigma,\infty,l^{p}}(\F,\rho_{\F,i})=\opdim_{\Sigma,\infty,l^{p}}(\F',\rho_{\F',i})\]
when  $\F,\F'$ are two $q$-dynamical filtrations and $\rho$ is a prodduct norm. Thus we can define
\[\opdim_{\Sigma,\infty,l^{p}}(q,\Phi,\rho)\]
to be these common numbers. After that, we will show that 
\[\opdim_{\Sigma,\infty,l^{p}}(q,\Phi',\rho)=\opdim_{\Sigma,\infty,l^{p}}(q,\Phi',\rho)\]
where $\Phi,\Phi'$ are two graphings. We can then call this number 
\[\opdim_{\Sigma,l^{p}}(q,\rho)\]

	We then show that 
\[f.\dim_{\Sigma,l^{p}}(S,\Phi,\rho)=f.\dim_{\Sigma,l^{p}}(S,\Phi,\rho'),\]
where $\rho,\rho'$ are two product norms. Putting all of this together implies that 
\[f.\dim_{\Sigma,l^{p}}(S,\Phi,\rho),\opdim_{\Sigma,l^{p},\infty}(q,\rho_{p,d_{i}})\]
do not depend on $q,S,\Phi$ or $\rho.$ We thus define $\dim_{\Sigma,l^{p}}(V,\R)$ to be any of these common numbers.

	The proof of all these facts will follow quite parallel to the proofs in \cite{Me}.
\section{Proof of Invariance}

	Our first Lemma is taken directly from \cite{Me}, Proposition 3.1 and will be used frequently without comment.

\begin{lemma}Let $Y$ be a separable Banach space with the $C$-bounded approximation property, and let $I$ be a countable directed set. Let $(Y_{\alpha})_{\alpha\in I}$ be a increasing net of subspaces of $Y$ such that
\[Y=\overline{\bigcup_{\alpha}Y_{\alpha}}.\]
Then  there are finite-rank maps $\theta_{\alpha}\colon Y\to Y_{\alpha}$  such that $\|\theta_{\alpha}\|\leq C$ and
\[\lim_{\alpha}\|\theta_{\alpha}(y)-y\|= 0\]
for all $y\in Y.$

\end{lemma}

	The next Lemma will be crucially used in passing between $\opdim$ and $\dim.$

\begin{lemma}\label{L:makeinv} Fix $1\leq p<\infty.$ Let $((X,\mu),\R,\Phi,V,W,q,\Sigma)$ be a quotient dimension tuple and $\F=(b_{j,\phi},W_{F,k})$ a  $(q,\Phi)$-dynamical filtration. Let $G\subseteq W$ be a finite-dimensional linear subspace and $\kappa>0.$ Let $\rho$ be a product norm and $\lambda>0$ so that $W$ has the $\lambda$-bounded approximation property.  Fix $M>\lambda.$  Then there is a $F\subseteq \Phi$ finite, $m\in \NN,$ $\delta,\varepsilon>0$ and linear maps 
\[L_{i}\colon l^{\infty}(\NN,l^{p}(d_{i}))\to B(W,l^{p}(d_{i})),\]
so that if $f\in l^{\infty}(\NN,l^{p}(d_{i})),T\in \Hom_{\R,l^{p}}(\F,F,m,\delta,\sigma_{i})$ and $B\subseteq \{1,\cdots,d_{i}\}$ has $|B|\geq (1-\varepsilon)d_{i},$ and $\rho_{l^{p}(d_{i})}(\chi_{B}(\alpha_{\F}(T)-f))<\varepsilon,$ then there is a $C\subseteq \{1,\cdots,d_{i}\}$ with $|C|\geq (1-\eta)d_{i}$ such that 
\[\|L_{i}(f)\|_{W\to l^{p}(C)}\leq M,\]
\[\|L_{i}(f)\big|_{G}-T\big|_{G}\|_{G\to l^{p}(C)}\leq \kappa.\]\end{lemma}
\begin{proof} Note that there is a $E\subseteq \Phi$ finite, $l\in \NN,$ so that 
\[\sup_{\substack{w\in G\\ \|w\|=1}}\inf_{\substack{v\in W_{E,l}\\ \|v\|=1}}\|v-w\|<\kappa.\]
Thus, we may assume that $G=W_{E,l}$ for some $E\subseteq \Phi$ finite, $l\in \NN.$ 

	Fix $\eta>0$ to be determined later. Let  $\theta_{F,k}\colon W\to W_{F,k}$ be such that 
\[\|\theta_{F,k}\|\leq \lambda,\]
\[\lim_{(F,k)}\|\theta_{F,k}(w)-w\|=0\mbox{ for all $w\in W$}.\]

	Choose $F,m$ sufficiently large so that
\[\|\theta_{F,m}\big|_{Y_{E,l}}-\id\big|_{Y_{E,l}}\|<\eta.\]

	Let $\B_{F,m}\subseteq F^{m}\times \{1,\cdots,m\}$ be such that $\{q(b_{\psi,j}:(\psi,j)\in \B_{F,m}\}$ is a basis for $V_{F,m}\colon=\Span\{q(b_{\psi,j}):(\psi,j)\in F^{m}\times \{1,\cdots,m\}\}.$ Define $\widetilde{L_{i}}\colon l^{\infty}(\NN,l^{p}(d_{i}))\to B(V_{F,m},l^{p}(d_{i}))$ by
\[\widetilde{L_{i}}(q(b_{\psi,j})=\sigma_{i}(\psi)f(j).\]

	We claim that if $\delta,\varepsilon>0$ are small enough, then for $f\in l^{\infty}(\NN,l^{p}(d_{i})),$ $T\in \Hom_{\R,l^{p}}(\F,F,m,\delta,\sigma_{i}),$ and $C\subseteq\{1,\cdots,d_{i}\}$ with $|C|\geq (1-\varepsilon)d_{i}$ and 
\[\rho(\chi_{C}(f-\alpha_{\F}(T)))<\varepsilon,\]
 there is a $B\subseteq \{1,\cdots,d_{i}\}$ so that $|B|\geq (1-\eta)d_{i}$ with
\[\|\widetilde{L_{i}}(f)\circ q\big|_{W_{E,l}}-T\big|_{W_{E,l}}\|_{W_{E,l}\to l^{p}(B)}\leq \eta.\]

	By finite-dimensionality, there is $D(F,m)>0$ so that if $v\in \ker(q)\cap W_{F,m}$ and $(\lambda_{\psi,r})\in \CC^{\B_{F,m}},$ then
\[\sup(\|v\|,|\lambda_{\psi,r}|)\leq D(F,m)\left\|v+\sum_{(\psi,r)\in \B_{F,m}}\lambda_{\psi,r}b_{\psi,r}\right\|.\]

	Thus if $x=v+\sum_{\psi,r}\lambda_{\psi,r}b_{\psi,r},$ and $C\subseteq B,$ then
\begin{equation}\label{E:almostinv}
\|\widetidle{L_{i}}(f)(q(x))-T(x)\|_{l^{p}(C)}\leq D(F,m)\delta+D(F,m)\sum_{\psi,r}\|\sigma_{i}(\psi)f(r)-T(b_{\psi,r})\|_{l^{p}(C)}.
\end{equation}

	Let $A\subseteq\{1,\cdots,d_{i}\}$ be such that $|A|\geq (1-\delta)d_{i},$ and
\[\|T(b_{\phi_{1}\cdots \phi_{m},j})-\sigma_{i}(\phi_{1})\cdots\sigma_{i}(\phi_{m})T(b_{\id,j})\|_{l^{p}(A)}<\delta,\]
and set $C=B\cap A.$ Then by $(\ref{E:almostinv})$ we have
\[\|\widetidle{L_{i}}(f)(q(x))-T(x)\|_{l^{p}(C)}\leq D(F,m)\delta+D(F,m)|F|^{m}m\delta+\sum_{(\psi,r)}\|f(r)-T(b_{\psi,r})\|_{l^{p}(C)},\]
so it suffices to choose $\delta,\varepsilon>0$ sufficiently small so that 
\[\delta+\varepsilon<\eta,\]
\[\delta<\frac{\eta}{2D(F,m)(1+|F|^{m}m)},\]
 and if $g\in l^{\infty}(\NN)$ has $\rho(g)<\varepsilon$ then
\[\sum_{(\psi,r)\in \B_{F,m}}g(r)<\frac{\eta}{2}.\]

	Now suppose that $\delta,\varepsilon>0$ are so chosen and set $L_{i}(f)=\widetilde{L_{i}}(f)\circ q\big|_{W_{F,m}}\circ \theta_{F,m},$ then if $T,f,C$ are as above and $w\in W_{E,l},$ then
\[\|L_{i}(f)(w)-T(w)\|_{l^{p}(C)}\leq(1+\eta)\|\theta_{F,m}(w)-w\|+\eta\|w\|\leq \eta(1+2\eta)\|w\|,\]
so it suffices to choose $\eta$ so that 
\[\eta(1+2\eta)<\kappa,\]
\[\lambda(1+\eta)<M.\]
\end{proof}

Our next lemma allows us to switch between two different pseudonorms.

\begin{lemma}\label{L:diffprodnorm} Fix $1\leq p<\infty.$ Let $((X,\mu),\R,\Phi,V,W,q,\Sigma)$ be a quotient dimension tuple and $\F=(b_{j,\phi},W_{F,k})$ a $(q,\Phi)$-dynamical filtration. Let $\F$ be a $(q,\Phi)$-dynamical filtration, $\rho$ a monotone product norm, and let $C>0$ so that $W$ has the $C$-bounded approximation property.

(a) If $C<M<\infty,$ then
\[f.\dim_{\Sigma,\infty,l^{p}}(\F,\Phi,\rho_{\F,i})=\opdim_{\Sigma,M,l^{p}}(\F,\Phi,\rho_{\F,i})\]
\[\underline{f.\dim}_{\Sigma,\infty,l^{p}}(\F,\Phi,\rho_{\F,i})=\underline{\opdim}_{\Sigma,M,l^{p}}(\F,\Phi,\rho_{\F,i}).\]

(b) If $\rho'$ is any other product norm, then for all $M>0,$ 
\[\opdim_{\Sigma,M,l^{p}}(\F,\Phi,\rho_{\F,i})=\opdim_{\Sigma,M,l^{p}}(\F,\Phi,\rho'_{\F,i})\]
\[\underline{\opdim}_{\Sigma,M,l^{p}}(\F,\Phi,\rho_{\F,i})=\underline{\opdim}_{\Sigma,M,l^{p}}(\F,\Phi,\rho'_{\F,i}).\]

\end{lemma}

\begin{proof}

(a) Let $\F=((b_{\phi,j}),(W_{F,l})_{F\subseteq \mathcal{W}(\Phi)\mbox{ finite},l\in \NN}).$ Let $A$ be such that 
\[\|b_{\phi,j}\|\leq A,\]

	Let $1>\varepsilon'>0.$ Find $k\in \NN,$ so that if $\|f\|_{\infty}\leq 1,$ and $f$ is supported on $\{n:n\geq k\},$ then $\rho(f)<\varepsilon'.$  Since $\rho$ induces a topology weaker than the norm topology, we can find a $\varepsilon'>\kappa>0$ so that $\rho(f)<\varepsilon',$ if $\|f\|_{\infty}\leq \kappa.$ 

	Let $\id\in E\subseteq \Phi$ be finite $\varepsilon'>\varepsilon>0,m\in \NN,\delta>0$ and $L_{i}\colon l^{\infty}(\NN,l^{p}(d_{i}))\to B(W,l^{p}(d_{i}))$ be as in the proceeding lemma for this $M,\kappa,$ and the finite-dimensional subspace $W_{\{\id\},k}.$ 

	Suppose $T\in \Hom_{\R,l^{p}}(\F,F,m,\delta,\sigma_{i}),f\in l^{\infty}(\NN,V_{i}),$ and $C\subseteq\{1,\cdots,d_{i}\}$ has $|C|\geq (1-\varepsilon)d_{i}$ with
\[\rho(\chi_{C}(f-\alpha_{\F}(T))<\varepsilon.\]
Let $B\subseteq\{1,\cdots,d_{i}\}$ be such that $|B|\geq(1-\kappa)d_{i},$ 
\[\|L_{i}(f)\|_{W\to l^{p}(B)}\leq M,\]
\[\|L_{i}(f)\big|_{W_{\{\id\},k}}-T\big|_{W_{\{\id\},k}}\|_{W_{\{\id\},k}\to l^{p}(B)}\leq \kappa.\]

	Then
\begin{align*}
\rho_{\F,i}(\chi_{B}(L_{i}(f)-T))&\leq M\varepsilon+\rho(j\to \|L_{i}(f)(b_{\{\id\},j})-T(b_{\{\id\},j})\|_{l^{p}(C)})\\
&\leq M\varepsilon'+A\varepsilon'.
\end{align*}

	Thus
\[\opdim_{\Sigma,M}(\F,F_{0},m_{0},\delta_{0},(M+A)\varepsilon',\Phi,\rho)\leq f.\dim_{\Sigma}(\F,F_{0},m_{0},\delta,\varepsilon,\Phi,\rho)\]
if $F_{0}\supseteq F,m_{0}\geq m,\delta_{0}<\delta.$ Thus
\[\opdim_{\Sigma,M}(\F,(M+A)\varepsilon',\Phi,\rho)\leq f.\dim_{\Sigma}(\F,\Phi,\rho),\]
and since $\varepsilon'$ was arbitrary, we are done. 

(b) This follows from compactness of $\|\cdot\|_{\infty}$ unit ball in the product topology.

\end{proof}

	We now proceed to show equality when we switching graphings, it is enough to handle the case of simply increasing the graphing.

\begin{lemma}\label{L:samepseudo} Fix $1\leq p<\infty.$ Let $((X,\mu),\R,\Phi,V,W,q,\Sigma)$ be a quotient dimension tuple and $\F=(b_{j,\phi},W_{F,k})$ a  $(q,\Phi)$-dynamical filtration.  Let $\Phi\subseteq \Phi'\subseteq [[\R]]$ with $\Phi'$ countable. Let  $\F'=((b_{j,\phi}'),W'_{F,k})$ be a $(q,\Phi')$ dynamical filtration extending $\F$. Suppose that  $\Sigma'$ is any sofic approximation  then 
\[\opdim_{\Sigma,\infty,l^{p}}(\F,\Phi,\rho)=\opdim_{\Sigma,\infty,l^{p}}(\F',\Phi',\rho),\]
\[\underline{\opdim}_{\Sigma,\infty,l^{p}}(\F,\Phi,\rho)=\underline{\opdim}_{\Sigma,\infty,l^{p}}(\F',\Phi',\rho).\]
\end{lemma}

\begin{proof} Let $C>0$ be such that for every $v\in V,$ there is a $w\in W$ so that $q(w)=v,$ and 
\[\|v\|\leq C\|w\|.\]

It is clear that
\[\opdim_{\Sigma',l^{p}}(\F,\Phi',\rho)\leq \opdim_{\Sigma,l^{p}}(\F,\Phi,\rho).\]

	For the opposite inequality, first note that for any  subset $E\subseteq *-\Alg(\Phi)\cap [[\R]]$  we have
\[\opdim_{\Sigma}(\F,\Phi,\rho)\leq \opdim_{\Sigma}(\F,E,\rho).\]

	Our assumptions imply that for any $\eta>0,$ for any $\psi\in [[\R]],$ there is a $\psi'\in [[\R]],$ with $v_{\psi'}\in *-\Alg(\Phi)$ such that 
\[\|v_{\psi}-v_{\psi'}\|_{2}<\eta.\]

	Fix $1\in F'\subseteq \Phi'$ finite, $\delta'>0$ and $m'\in \NN.$ Let $\eta>0$ to be determined later. By our above observation, we can find a finite subset $E\subseteq *-\Alg(\Phi)\cap [[\R]]$ such that for every $\phi'\in F',$ there is a $\phi\in E$ so that 
\[\|\phi_{1}\cdots \phi_{m}a_{j}-\phi_{1}'\cdots \phi_{m}'a_{j}\|<\frac{\delta}{C}\mbox{ for all $1\leq j\leq m,$ and $\phi_{1}',\cdots,\phi_{m}'\in F'$},\]
\[\|\phi_{1}\cdots \phi_{m}-\phi_{1}'\cdots\phi_{m}'\|_{2}<\eta \mbox{ for all $\phi_{1}',\cdots,\phi_{m}'\in F'$}.\]

	We use $\mathcal{W}(\Phi)$ for all finite products of elements in $\Phi\cup \Phi^{*}\cup \id,$ and we use $\mathcal{W}_{m}(\Phi)$ for $[\Phi\cup \Phi^{*}\cup \id]^{m}.$ 
	Thus we can find a finite subset $E\subseteq F\subseteq \mathcal{W}(\Phi),$ and an $m\in \NN$ and $w_{\phi_{1}'\cdots \phi_{m}',j}\in \ker(q)\cap W_{F,m}$ so that 
\[\|b_{\phi_{1}'\cdots\phi_{m}',j}-b_{\phi_{1}\cdots\phi_{m},j}-w_{\phi_{1}'\cdots\phi_{m}',j}\|<\delta.\]

	We may assume that $F,m$ are sufficiently large so that
\[\sup_{w\in \Ball(W_{F',m'}\cap \ker(q))}\inf_{v\in W_{F,m}}\|w-v\|<\delta,\]
\[E\subseteq \mathcal{W}_{m}(\Phi).\]

	Let $\delta>0$ which will depend upon $\delta',F',m'$ in a manner to be determined later.

	Fix $T\in \Hom_{\R,l^{p}}(\F,F,m,\delta,\sigma_{i})$ and suppose $A$ is such that
\[\|T(b_{j,\phi_{1}\cdots,\phi_{m}})-\sigma_{i}(\phi_{1})\cdots\sigma_{i}(\phi_{m})T(b_{j,\id})\|_{l^{p}(A)}<\delta\]
for all $\phi_{1}',\cdots,\phi_{m}'\in F'.$ Let $C$ be the set of $j$ in $\{1,\cdots,d_{i}\}$ so that whenever $\phi_{1},\cdots,\phi_{m}\in F,$ then  
\begin{align*}
j\notin \dom(\sigma_{i}(\phi_{1})\cdot \sigma_{i}(\phi_{m}))\Delta \dom(\sigma_{i}(\phi_{1}')\cdot \sigma_{i}(\phi_{m}'))\\
\sigma_{i}(\phi_{m})^{-1}\cdot \sigma_{i}(\phi_{1})^{-1}(j)=\sigma_{i}(\phi_{m}')^{-1}\cdot \sigma_{i}(\phi_{1})^{-1}(j), \mbox{ if either side is defined}.
\end{align*}

	If $\eta$ is sufficiently small, then for all large $i,$ $|C|\geq (1-\delta')d_{i}.$ 

	Thus for all $1\leq j\leq m$ and $\phi_{1},\cdots,\phi_{m}\in F$ we have
\begin{align*}
&\|T(b_{\phi_{1}'\cdots,\phi_{m}',j})-\sigma_{i}(\phi_{1}')\cdots\sigma_{i}(\phi_{m}')T(b_{\id,j})\|_{l^{p}(A\cap C)}\\
&=\|T(b_{\phi_{1}\cdots,\phi_{m},j})-\sigma_{i}(\phi_{1})\cdots\sigma_{i}(\phi_{m})T(b_{\id,j})\|_{l^{p}(A\cap C)}\\
&\leq \delta'\|w_{\phi_{1}'\cdots\phi_{m}',j}\|+\delta',
\end{align*}
also our assumptions on $F',m'$ ensure that 
\[\|T\big|_{\ker(q)\cap W_{F',m'}}\|\leq \delta+\delta'.\]

	Thus if $\delta$ is sufficiently small, we may ensure that $T\in \Hom_{\R,l^{p}}(\F,F',m',2\delta',\sigma_{i}).$ So for any $\varepsilon>0,$ we have 
\[\opdim_{\Sigma,l^{p}}(\F,\varepsilon,\Phi,\rho)\leq \opdim_{\Sigma',l^{p}}(\F',F',m',\delta',\varepsilon,\Phi',\rho).\]
Since $F',m',\delta',\varepsilon'$ were arbitrary, we see that 
\[\opdim_{\Sigma,l^{p}}(F,\Phi,\rho)\leq \opdim_{\Sigma',l^{p}}(\F',\Phi',\rho).\]
\end{proof}

	We now show that $\opdim_{\Sigma,\infty}(\F,\Phi,\rho_{\F,i})$ only depends upon $\Phi$ and the quotient map $q.$ Because of Lemmas \ref{L:makeinv},\ref{L:diffprodnorm},\ref{L:samepseudo} for any other $(q,\Phi)$-dynamical filtration $\F'$
\[\opdim_{\Sigma,\infty,l^{p}}(\F,\Phi,\rho_{\F,i})=\opdim_{\Sigma,l^{p}}(\F',\Phi,\rho_{\F,i}),\]
so the only difficulty is in switching $\rho_{\F,i}$ to $\rho_{\F',i}.$ To do this, we will have to investigate how much our definition of dimension depends on the choice of pseudonorm.

\begin{definition} \emph{ Let $((X,\mu),\R,\Phi,V,W,q,\Sigma)$ be a quotient dimension tuple and $\F=(b_{j,\phi},W_{F,k})$ a  $(q,\Phi)$-dynamical filtration. Let $\rho_{i},q_{i}$ be two sequence of pseudonorms on $B(W,l^{p}(d_{i})),$ we say that $\rho_{i}$ is } $(\F,\Sigma)$ weaker than $q_{i}$ \emph{and write $\rho_{i}\preceq_{\F,\Sigma}q_{i},$ if for every $\varepsilon'>0,$ there are $\varepsilon,\delta>0,m,i_{0}\in \NN,$  $F\subseteq \Phi$ finite, and linear maps $L_{i}\colon B(W,l^{p}(d_{i}))\to B(W,l^{p}(d_{i}))$ for $i\geq i_{0},$ so that if $G$ is a linear subspace of $B(W,l^{p}(d_{i}))$ and $\Hom_{\R,l^{p}}(\F,F,m,\delta,\sigma_{i})\subseteq_{\varepsilon,q_{i}}G,$ then $\Hom_{\R,l^{p}}(\F,F,m,\delta,\sigma_{i})\subseteq_{\varepsilon',\rho_{i}}L_{i}(G).$} \end{definition}

\begin{lemma}\label{L:weaker} Let $((X,\mu),\R,\Phi,V,W,q,\Sigma)$ be a  quotient dimension tuple and $\F=(b_{j,\phi},W_{F,k})$ a  $(q,\Phi)$-dynamical filtration. 

(a) If $\rho_{i},q_{i}$ are two sequence of pseudonorms on $B(W,l^{p}(d_{i}))$ and $\rho_{i}\preceq_{\F,\Sigma}q_{i},$ then
\[\opdim_{\Sigma,\infty,l^{p}}(\F,\Phi,\rho_{i})\leq \opdim_{\Sigma,\infty,l^{p}}(\F,\Phi,q_{i}),\]
\[\underline{\opdim}_{\Sigma,\infty,l^{p}}(\F,\Phi,\rho_{i})\leq \underline{\opdim}_{\Sigma,\infty,l^{p}}(\F,\Phi,q_{i}),\]

(b) Let $\F'$ be another $q$-dynamical filtration, then $\rho_{\F',i}\preceq_{\F,\Sigma}\rho_{\F,i}.$

\end{lemma} 

\begin{proof}
(a) Follows directly from the definitions.

(b) Let $\F=((b_{\phi,j}),(W_{E,l})),\F'=((b_{\phi,j}'),(W_{E,l}')).$ Let $C>0$ be such that $W$ has the $C$-bounded approximation property, and

\[\|b_{\phi,j}\|\leq C,\]
\[\|b_{\phi,j}'\|\leq C.\]

	Fix $\varepsilon'>0>$ Choose $k\in \NN,$ so that if $f$ is supported on $\{n:n\geq k\}$ and $\|f\|_{\infty}\leq 1,$ then $\rho(f)<\varepsilon,$ and let $\varepsilon'>\kappa>0$ be such that $\rho(f)<\varepsilon'$ if $\|f\|_{\infty}\leq \kappa.$ Let $F,m,\delta,\varepsilon,$ and $L_{i}\colon l^{\infty}(\NN,l^{p}(d_{i}))\to B(W,l^{p}(d_{i}))$ be as Lemma $\ref{L:makeinv}$ for this $\kappa,$ $M=2C$ and the finite-dimensional subspace $W'_{\{\id\},k}.$  Define $\alpha_{F,i}\colon B(W,l^{p}(d_{i}))\to l^{\infty}(\NN,l^{p}(d_{i}))$ by $\alpha_{\F}(T)(n)=T(b_{\id,n}).$ Set $\widetilde{L_{i}}(T)=L_{i}(\alpha_{\F}(T)).$ We may assume that $m\geq k.$

	Suppose that $\Hom_{\R,l^{p}}(\F,F,m,\delta,\sigma_{i})\subseteq_{\varepsilon,\rho_{\F,i}}G,$ then by Lemma $\ref{L:makeinv},$ for every $T\in \Hom_{\R,l^{p}}(\F,F,m,\delta,\sigma_{i})$ we can find an $S\in B(W,l^{p}(d_{i}))$ and a $C\subseteq\{1,\cdots,d_{i}\}$ with $|C|\geq (1-\kappa)d_{i}$ so that 
\[\|\widetilde{L_{i}}(S)\|_{W\to l^{p}(C)}\leq 2C,\]
\[\|\widetidle{L_{i}}(S)\big|_{W_{\{\id\},k}'}-T\big|_{W_{\{\id\},k}'}\|<\kappa.\]

	Thus
\begin{align*}
\rho_{\F,i}(m_{\chi_{C}}(\widetilde{L_{i}}(S)-T))&\leq (2C+1)\varepsilon'+\rho(\chi_{j\leq k}\|\widetidle{L_{i}}(S)(b_{\id,j})-T(b_{\id,j})\|_{l^{p}(C)})\\
&\leq (2C+1)\varepsilon'+C\varepsilon'
\end{align*}

	This proves $(b).$ 
 
\end{proof}

\begin{cor}\label{L:indepfilt} Fix $1\leq p<\infty.$ Let $((X,\mu),\R,\Phi,V,W,q,\Sigma)$ be a quotient dimension tuple, and $\F$ a $(q,\Phi)$-dynamical filtration. If $\F'$ is another $(q,\Phi)$-dynamical filtration, and $\rho,\rho'$ are two product norms, then
\[\opdim_{\Sigma,\infty,l^{p}}(\F,\Phi,\rho)=\opdim_{\Sigma,\infty,l^{p}}(\F',\Phi,\rho')\]
\[\underline{\opdim}_{\Sigma,\infty,l^{p}}(\F,\Phi,\rho)=\underline{\opdim}_{\Sigma,\infty,l^{p}}(\F',\Phi,\rho').\]
\end{cor}

\begin{proof} If we combine Lemmas $\ref{L:diffprodnorm}$-$\ref{L:weaker}$ we obtain 
\[\opdim_{\Sigma,\infty,l^{p}}(\F,\Phi,\rho)\leq \opdim_{\Sigma,\infty,l^{p}}(\F',\Phi,\rho')\]
the result follows by symmetry.

\end{proof}

Because of the above corollary, we may define
\[\dim_{\Sigma,l^{p}}(q,\Phi)=\opdim_{\Sigma,\infty,l^{p}}(\F,\Phi,\rho),\]
\[\underline{\dim}_{\Sigma,l^{p}}(q,\Phi)=\underline{\opdim}_{\Sigma,\infty,l^{p}}(\F,\Phi,\rho),\]
where $\F,\rho$ are as in the statement of the corollary. 
\begin{lemma}\label{L:limit} Let $((X,\mu),\R,\Phi,V,\Sigma)$ be a dimension tuple, and let $\rho$ be a product norm. Let $S$ be a dynamically generating sequence in $V.$  Then
\[f.\dim_{\Sigma,l^{p}}(S,\Phi,\rho)=\sup_{\varepsilon>0}\liminf_{(F,m,\delta)}\limsup_{i\to \infty}f.\dim_{\Sigma,l^{p}}(S,F,m,\delta,\varepsilon,\Phi,\rho),\]
\[\underline{f.\dim}_{\Sigma,l^{p}}(S,\Phi,\rho)=\sup_{\varepsilon>0}\limsup_{(F,m,\delta)}\liminf_{i\to \infty}f.\dim_{\Sigma,l^{p}}(S,F,m,\delta,\varepsilon,\Phi,\rho).\]
\end{lemma}

\begin{proof} Let $S=(a_{j})_{j=1}^{\infty},$ and $C>0$ so that 
$\|a_{j}\|\leq C$
for all $j.$ 

 Fix $\varepsilon>0,$ and choose $k\in \NN$ so that $\rho(f)<\varepsilon$ if $\|f\|_{\infty}\leq 1$ and $f$ is supported on $\{n:n\geq k\}.$ Fix $F\subseteq \Phi$ finite, a natural number $m\geq k$ and $\delta>0.$ Then if $F'\supseteq F$ is a finite subset of $\Phi,$ $m'\geq m$ is a natural number and $0<\delta'<\delta,$ then $\Hom_{\R,l^{p}}(S,F',m',\delta',\sigma_{i})\subseteq \Hom_{\R,l^{p}}(S,F,m,\delta,\sigma_{i}).$ Further, for $f\in l^{\infty}(\NN,l^{p}(d_{i}))$ with $\|f\|_{l^{\infty}(\NN,l^{p}(d_{i}))}$ we have  
\[\rho(\chi_{j\leq m}f(j)-\chi_{j\leq m'}f(j))<\varepsilon.\]

	Thus
\[d_{2\varepsilon}(\alpha_{S}(\Hom_{\R,l^{p}}(S,F',m',\delta',\sigma_{i})),\rho)\leq d_{\varepsilon}(\alpha_{S}(\Hom_{\R,l^{p}}(S,F,m,\delta,\sigma_{i}),\rho).\]

	This implies that 
\[f.\dim_{\Sigma,l^{p}}(S,2\varepsilon,\Phi,\rho)\leq f.\dim_{\Sigma,l^{p}}(S,F,m,\delta,\varepsilon,\Phi,\rho).\]

Since $F,m,$ were arbitrarily large, $\delta>0$ was arbitrarily small we see that 
\[f.\dim_{\Sigma,l^{p}}(S,2\varepsilon,\Phi,\rho)\leq \liminf_{(F,m,\delta)}f.\dim_{\Sigma,l^{p}}(S,F,m,\delta,\varepsilon,\Phi,\rho)\]
and taking the supremum over $\varepsilon>0$ completes the proof.
\end{proof}

\begin{lemma} Let $((X,\mu),\R,\Phi,V,W,q,\Sigma)$ be a quotient dimension tuple. Let $S$ be a dynamically generating sequence in $V.$ Then for any product norm $\rho$ we have
\[\dim_{\Sigma,l^{p}}(q,\Phi,\rho)=f.\dim_{\Sigma,l^{p}}(S,\Phi,\rho),\]
\[\underline{\dim}_{\Sigma,l^{p}}(q,\Phi,\rho)=\underline{f.\dim}_{\Sigma,l^{p}}(S,\Phi,\rho).\]

\end{lemma}

\begin{proof} 
Let $S=(a_{j})_{j=1}^{\infty},$ and let $\F=((b_{\phi,j}),(W_{E,l}))$ be a $(q,\Phi)$-dynamical filtration such that $q(b_{\id,j})=a_{j}.$ Let $C>0$ be such that 
\[\|a_{j}\|\leq C,\]
\[\|b_{\phi,j}\|\leq C,\]
\[\|q\|\leq C,\]
\[\mbox{ for every $v\in V,$ there is a $w\in W$ such that $q(w)=v,$ and $\|w\|\leq C\|v\|,$}\]
\[\mbox{ $W$ has the $C$-bounded approximation property}.\]

	Let $\theta_{F,k}\colon W\to W_{F,k}$ be such that $\|\theta_{F,k}\|\leq C$ and 
\[\lim_{(F,k)}\|\theta_{F,k}(w)-w\|=0\mbox{ for all $w\in W$.}\]

	We first show that 
\[f.\dim_{\Sigma,l^{p}}(\F,\Phi,\rho)\geq f.\dim_{\Sigma}(S,\Phi,\rho).\]

	Fix $\varepsilon>0,$ and choose $k\in \NN,$ so that $\rho(f)<\varepsilon$ if $\|f\|_{\infty}\leq 1$ and $f$ is supported on $\{n:n\geq k\}.$ Choose $\kappa>0$ so that $\rho(f)<\varepsilon$ if $\|f\|_{\infty}\leq \kappa.$ Let $\id\in E\subseteq \Phi$ finite and $l\in \NN,$ so that if $F\supseteq E,m\geq l$ then
\[\|\theta_{F,m}(b_{\id,j})-b_{\id,j}\|\leq \kappa\]
for all $1\leq j\leq k.$ 

	Fix $E\subseteq F\subseteq \Phi$ finite $l\leq m\in \NN,$ and $\delta>0.$ We claim that we can find a $F'\subseteq \Phi$ finite, and $m'\in \NN$ and $\delta'>0$ so that
\[\Hom_{\R,l^{p}}(S,F',m',\delta',\sigma_{i})\circ q\big|_{W_{F',m'}}\circ \theta_{F',m'}\subseteq \Hom_{\R,l^{p}}(\F,F,m,\delta,\sigma_{i})_{C^{2}}.\]

	If $T\in \Hom_{\R,l^{p}}(S,F',m',\delta',\sigma_{i}),$ $B\subseteq\{1,\cdots,d_{i}\}$ is as in the definition of $\Hom_{\R,l^{p}}(S,F',m',\delta',\sigma_{i}),$ if $1\leq j\leq m$ and $\phi_{1},\cdots,\phi_{m}\in F$ then
\begin{align*}
&\|T\circ q\circ \theta_{F',m'}(b_{\phi_{1}\cdots \phi_{m},j})-\sigma_{i}(\phi_{1})\cdots \sigma_{i}(\phi_{m})T(q(\theta_{F',m'}(b_{\id,j}))\|_{l^{p}(B)}\\
&\leq C\|\theta_{F',m'}(b_{\phi_{1}\cdots \phi_{m},j})-b_{\phi_{1}\cdots\phi_{m},j})\|_{l^{p}(B)}+C\|\theta_{F',m'}(b_{\id,j})-b_{\id,j}\|_{l^{p}(B)}+\delta'.
\end{align*}

	For $w\in \ker(q)\cap W_{F,m}$ we have
\[\|T\circ q\circ \theta_{F',m'}(w)\|\leq C\|\theta_{F',m'}(w)-w\|,\]
so it suffices to choose $\delta'<\delta$ and then $F',m'$ large so that for all $1\leq j\leq m,\psi\in F^{m},$
\[C\|\theta_{F',m'}(b_{\psi,j})-b_{\psi,j})\|+C\|\theta_{F',m'}(b_{\id,j})-b_{\id,j}\|,\]
\[\|\theta_{F',m'}\big|_{W_{F,m}}-W_{F,m}\|<\frac{\delta}{C}.\]

	Suppose that $F',m',\delta'$ are so chosen, and that $m'\geq k.$ If $T\in \Hom_{\R,l^{p}}(S,F',m',\delta),$ then
\begin{align*}
\rho(\alpha_{\F}(T\circ q\circ \theta_{F',m'})-\alpha_{S}(T))&\leq (C^{2}+1)\varepsilon+\rho(\chi_{j\leq k}\|T\circ q\circ \theta_{F',m
'}(b_{\id,j})-T\circ q(b_{\id,j})\|)\\
&\leq (C^{2}+C+1)\varepsilon.
\end{align*}

Thus
\[f.\dim_{\Sigma,l^{p}}(S,F,m,\delta,(C^{2}+C+2)\varepsilon,\Phi,\rho)\leq \opdim_{\Sigma},l^{p}(\F,F',m',\delta',\varepsilon,\rho)_{C},\]
since $F',m'$ were arbitrarily large and $\delta'$ arbitrarily small we have
\[f.\dim_{\Sigma,l^{p}}(S,F,m,\delta,\varepsilon,\Phi,\rho)\leq \opdim_{\Sigma}(\F,\varepsilon,\rho)_{C},\]
taking the limit supremum over $(F,m,\delta)$ and then the supremum over $\varepsilon>0$ we find that 
\[f.\dim_{\Sigma,l^{p}}(S,\Phi,\rho)\leq \dim_{\Sigma,l^{p}}(q,\Phi,\rho).\]

	For the opposite inequality, let $1>\varepsilon>0,$ and let $k,E,l,\kappa$ be as in the first half of the proof. Fix $E\subseteq F\subseteq \Phi$ finite, $m\geq \max(k,l)$ and $0<\delta<\kappa.$

	By Lemma 3.8 in \cite{Me}, choose a $0<\delta''<\delta$ a $F\subseteq F'\subseteq \Phi$ finite, a $m\leq m'\in \NN$ so that if  $E$ is Banach space and 
\[T\colon W_{F',m'}\to E\]
is a contraction with
\[\|T\big|_{\ker(q)\cap W_{F',m'}}\|\leq \delta''\]
then there is a linear map $A\colon V_{F,m}\to E$ with $\|A\|\leq 2C$ and
\[\|T(b_{\psi,j})-A(\psi a_{j})\|<\delta\mbox{ for all $1\leq j\leq m,$ and $\psi\in F^{m}$}\]

	Let $F',m'$ be as above and $T\in \Hom_{\R,l^{p}}(\F,F',m',\delta',\sigma_{i})$ and chose $S$ as in the preceding paragraph. Let $B\subseteq \{1,\cdots,d_{i}\}$ be as in the definition for $\Hom_{\R,l^{p}}(\F,F',m',\delta',\sigma_{i}).$  Then for all $1\leq j\leq m$ and $\phi_{1},\cdots,\phi_{m}\in F$ we have
\begin{align*}
\|A(\phi_{1}\cdots \phi_{m}a_{j})-\sigma_{i}(\phi)\cdots \sigma_{i}(\phi_{m})A(a_{j})\|_{l^{p}(B)}&\leq 2\delta+\|T(b_{\phi_{1}\cdots \phi_{m},j})-\sigma_{i}(\phi)\cdots \sigma_{i}(\phi_{m}T(b_{\id,j})\|_{l^{p}(B)}\\
&\leq 2\delta+\delta'\\
&<3\delta.
\end{align*}
Further

\[\rho(\alpha_{S}(A)-\alpha_{\F}(T))\leq (2C^{2}+C)\varepsilon+\rho(\chi_{j\leq k}\|A(a_{j})-T(b_{\id,j})\|)\leq (2C^{2}+C+1)\varepsilon.\]

	Thus
\[f.\dim_{\Sigma,l^{p}}(\F,(2C^{2}+C+2)\varepsilon,\Phi,\rho)\leq f.\dim_{\Sigma,l^{p}}(S,F,m,3\delta,\varepsilon,\Phi,\rho)\]
so taking a limit infimum over $(F,m,\delta)$ and then a supremum over $\varespilon$ completes the proof.

\end{proof}

Putting together all our Lemmas imply that  we can set
\[\dim_{\Sigma,l^{p}}(V,\R)=\dim_{\Sigma,l^{p}}(S,\Phi,\rho),\]
\[\underline{\dim}_{\Sigma,l^{p}}(V,\R)=\underline{\dim}_{\Sigma,l^{p}}(S,\Phi,\rho),\]
and this is independent of our choice of $S,\Phi,\rho.$

\section{Properties of Extended von Neumann Dimension}

\begin{definition}\emph{Let $(\R,X,\mu)$ be as above and $V$ a Banach space representation of $\R.$ If $v\in V,$ then since $(X,\mu)$ is standard there is a unique (up to measure zero) set $A$ such that $\id_{A}v=v$ and $\id_{A^{c}}v=0.$ We call $A$ the support of $v,$ and denote it by $\supp v.$}\end{definition}

	The following inequality is frequently useful, and will be used to great extent in Section \ref{S:cohom}.

\begin{proposition} Let $((X,\mu),\R,V,\Phi,\Sigma)$ be a dimension tuple. Let $S=(a_{j})_{j=1}^{\infty}$ be a dynamically generating sequence in $V,$ then for any sofic approximation $\Sigma,$ and $1\leq p<\infty,$
\[\dim_{\Sigma,l^{p}}(V,\R)\leq \sum_{j=1}^{\infty}\mu(\supp a_{j}).\]
\end{proposition}

\begin{proof} Let $A_{j}=\supp a_{j}.$  Fix $\varepsilon>0,$ let $F\subseteq \Phi$ be finite, $m\in \NN,\delta>0,$ if  $F$ is sufficiently large, then there is a $B_{j}\subseteq X$ measurable with $\id_{B_{j}}\in F^{m}$ so that 
\[\|\id_{B_{j}}a_{j}-a_{j}\|<\varepsilon,\]
\[\mu(B_{j}\Delta A_{j})<\delta.\]

	Thus for all large $i,$  and for all $T\in \Hom_{\R,l^{p}}(S,F,m,\delta,\sigma_{i})$ we can find a set $C\subseteq \{1,\cdots,d_{i}\}$ with $|C|\geq(1- 2\delta(1+m))d_{i}$ so that 
\[\|T(a_{j})-\sigma_{i}(\id_{A_{j}})T(a_{j})\|_{l^{p}(C)}<\delta.\]
	So if $\delta$ is sufficiently small (depending only upon $\varepsilon,m$) we have shown that 
\[(T(a_{1}),\cdots,T(a_{m}))\subseteq_{\varepsilon} \bigoplus_{j=1}^{n}\sigma_{i}(\id_{A_{j}})(l^{p}(d_{i})),\]
so for all large $i,$ 
\[\frac{1}{d_{i}}d_{\varepsilon}(\Hom_{\R,l^{p}}(S,F,m,\delta,\sigma_{i}))\leq \frac{1}{d_{i}}\sum_{j=1}^{n}\Tr(\sigma_{i}(\id_{A_{j}}))\to \sum_{j=1}^{m}\mu(A_{j}).\]

	Thus
\[f.\dim_{\Sigma}(S,F,m,\varepsilon,\delta,\sigma_{i})\leq \sum_{j=1}^{\infty}\mu(A_{j}),\]
since the above is true for all $F,m,$ sufficiently large and $\delta$ sufficiently small (depending only on $\varepsilon$) the proof is complete. 

\end{proof}

\begin{proposition} Let $((X,\mu),\R,\Phi,V,\Sigma)$ be a dimensional tuple, and let $W$ be another representation of $\R.$ If $T\colon W\to V$ is a bounded equivariant map with dense image, then
\[\dim_{\Sigma,l^{p}}(V,\R)\leq \dim_{\Sigma,l^{p}}(W,\R).\]
\end{proposition}
\begin{proof} If if $S$ is a dynamically generating sequence in $W,$ then $T\circ S$ is a dynamically generating for $W.$ If $\phi \in \Hom_{\R,l^{p}}(T\circ S,F,m,\delta,\sigma_{i}),$ then $\phi\circ T\in \Hom_{\R,l^{p}}(S,F,m,\delta,\sigma_{i})$ and
\[\alpha_{S}(\phi\circ T)=\alpha_{T\circ S}(\phi)\]
we are done.

\end{proof}

	We also handle how dimension behaves under compressions. This implies in particular that dimension is in fact invariant under \emph{weak isomorphism}.

\begin{proposition}\label{P:compress} Fix $1\leq p<\infty.$ Let $((X,\mu),\R,\Phi,V,\Sigma)$ be a dimensional tuple with $\R$ ergodic and $(X,\mu)$ diffuse. For a measurable $A\subseteq X,$ let $\Sigma_{A}$ be defined by $\sigma_{A,i}(\phi)=\sigma_{i}(\id_{A})\sigma_{i}(\phi)\sigma_{i}(\id_{A}).$ Then
\[\mu(A)\dim_{\Sigma,l^{p}}(\id_{A}V,\R_{A})=\dim_{\Sigma,l^{p}}(V,\R)\]
\[\mu(A)\underline{\dim}_{\Sigma,l^{p}}(\id_{A}V,\R_{A})=\underline{\dim}_{\Sigma,l^{p}}(V,\R)\]
\end{proposition}

\begin{proof} We will handle the case of $\dim$ only. Let $S_{A}=(a_{j})_{j=1}^{\infty}$ be a dynamically generating sequence for $V_{A}.$ Find $\psi_{1},\cdots,\psi_{k}\in [[\R]]$ with $\psi_{1}=\id_{A},\dom(\psi_{j})= A$ for $1\leq j \leq n,\dom(\psi_{n})\subseteq A,$ and up to sets of measure zero,
\[X=\bigsqcup_{j=1}^{k}\ran(\psi_{j}).\]
	Set $A_{j}=\psi_{j}(A).$  Let $S$ be an enumeration of $(\psi_{k}a_{j})_{j,k}.$ 

	We will first prove that $\dim_{\Sigma,l^{p}}(V,\R)\leq \mu(A)\dim_{\Sigma_{A},l^{p}}(V_{A},\R_{A})$ when $\mu(A)=1/n.$

	It is easy to see that
\[\dim_{\Sigma_{A_{j}},l^{p}}(\id_{A_{j}}V,\R_{A_{j}})\]
is independent of $j.$ For $T\colon V\to l^{p}(d_{i}),$ let $T_{A_{j}}\colon V_{A_{j}}\to l^{p}(\sigma(\id_{A_{j}})(\{1,\cdots,d_{i}\}))$ be given by
\[T_{A_{j}}(x)=\sigma_{i}(\id_{A_{j}})T(x).\]

	Fix $\varepsilon'>0,$ and let $\varepsilon>0$ depend upon $\varepsilon'$ in a manner to be determined later. 

	Given $F\subseteq \Phi_{A},m\in \NN,\delta>0,$ there is a $F'\subseteq \Phi,m'\in \NN,\delta'>0$ so that $T\in \Hom_{\R,l^{p}}(S,F',m',\delta',\sigma_{i})$ implies $T_{A}\in \Hom_{\R,l^{p}}(S_{A},F,m,\delta,\sigma_{i,A}).$ If we choose $F',m',\delta'$ appropriately and 
\[\alpha_{S_{A}}(\Hom_{\R_{A},l^{p}}(S_{A},F,m,\delta,\sigma_{i,A}))\subseteq_{\varepsilon,\|\cdot\|_{p}}W,\]

	then
\[\alpha_{S}(\Hom_{\R,l^{p}}(S,F',m',\delta',\sigma_{i}))\subseteq_{\varepsilon',\|\cdot\|_{p}}\left\{\sum_{k=1}^{n}\sigma_{i}(\psi_{k})\xi:\xi \in W\right\}.\]

	Since 
\[\frac{\tr(\sigma_{i}(\id_{A})}{d_{i}}\to \frac{1}{n},\]
we find that 
\[\dim_{\Sigma,l^{p}}(V,\R)\leq \frac{1}{n}\dim_{\Sigma,l^{p}}(V_{A},\R_{A}).\]
\[\underline{\dim}_{\Sigma,l^{p}}(V,\R)\leq \frac{1}{n}\underline{\dim}_{\Sigma,l^{p}}(V_{A},\R_{A}).\]

	We now show that  $\dim_{\Sigma,l^{p}}(V,\R)\leq \mu(A)\dim_{\Sigma_{A},l^{p}}(V_{A},\R_{A})$ for general $A$ (not necessarily with $\mu(A)=1/n$).
	Fix $F\subseteq [[\R]]$ finite $r\in \NN,\delta>0.$ Fix $\kappa>0$ which will depend upon $\delta$ in a manner to be determined. Let
\[F'\supseteq\{\psi_{i}^{-1}\phi \psi_{q}:1\leq i,q\leq k,\phi\in F\}\]
	Fix $r'\in \NN,\delta'>0$ which will depend upon $r,\delta$ in a manner to be determined shortly. Suppose that $T\in \Hom_{\R_{A},l^{p}}(S_{A},F',r',\delta',\sigma_{i}),$
	define 
\[\widetidle{T}(x)=\sum_{j=1}^{k}\sigma_{i}(\psi_{i})T(\psi_{i}^{-1}x).\]

	Then
\[\|\widetidle{T}\|\leq M,\]
where $M>0$ is some constant.
	
	Choose $C\subseteq \{1,\cdots,d_{i}\}$ of cardinality at least $(1-\delta')d_{i}$ for $T$ as in the definition of $\Hom_{\R_{A},l^{p}}(S_{A},F',r',\delta',\sigma_{i}).$ It is easy to see that if $F',r'$ are sufficently large and $\delta'$ is sufficently small, then
\[\|T(\psi_{i}^{-1}\phi\psi_{q}\psi_{q}^{-1} a_{l})-\sigma_{i}(\psi_{i}^{-1}\phi\psi_{q}^{-1})T(\psi_{q}^{-1}a_{l})\|_{l^{p}(C)}<\kappa.\]

	We have 
\[\psi_{j}^{-1}\phi=\sum_{q=1}^{k}\psi_{j}^{-1}\phi\psi_{q}\psi_{q}^{-1},\]
hence
\[T(\psi_{j}^{-1}\phi a_{l})=\sum_{q=1}^{k}T(\psi_{j}^{-1}\phi \psi_{q}\psi_{q}^{-1}a_{l}),\]
so
\[\left\|T(\psi_{j}^{-1}\phi a_{l})-\sum_{q=1}^{k}\sigma_{i}(\psi_{j}^{-1}\phi \psi_{q})T(\psi_{q}^{-1}a_{l})\right\|_{l^{p}(C)}<\frac{\delta}{k},\]
if $\kappa>0$ is sufficiently small. Since
\[\sum_{i=1}^{k}\psi_{j}\psi_{j}^{-1}\phi \psi_{q}=\phi\psi_{q},\]
if $i$ is sufficiently large, then we can find a set $C'\subseteq \{1,\cdots,d_{i}\}$ of size at least $\delta'$ so that of $C'$ we have
\[\sum_{i=1}^{k}\sigma_{i}(\psi)\sigma_{i}(\psi_{j}^{-1}\phi\psi_{q})=\sigma_{i}(\phi)\sigma_{i}(\psi_{q}).\]
Then
\begin{align*}
&\left\|\widetidle{T}(\phi a_{l})-\sigma_{i}(\phi)T(a_{l})\right\|_{l^{p}(C\cap C')}
\\&=\left\|\sum_{i=1}^{k}\sigma_{i}(\psi_{j})T(\psi_{j}^{-1}\phi a_{l})-\sum_{i=1}^{k}\sigma_{i}(\phi)\sigma_{i}(\psi_{j})T(\psi_{j}^{-1}a_{l})\right\|_{l^{p}(C\cap C')}\\
&\leq \delta+\left\|\sum_{1\leq j,q\leq k}\sigma_{i}(\psi_{j})\sigma_{i}(\psi_{j}^{-1}\phi\psi_{q})T(\psi_{q}^{-1}a_{l})-\sum_{q=1}^{k}\sigma_{i}(\psi_{Q})\sigma_{i}(\psi_{q})T(\psi_{j}^{-1}a_{l})\right\|_{l^{p}(C\cap C')}\\
&=\delta.
\end{align*}

	Thus $\widetidle{T}\in \Hom_{\R,l^{p}}(S,F,r,\delta,\sigma_{i})_{M}.$ Further, since $\psi_{1}=\id_{A},$ 
\[\sum_{j=1}^{k}\id_{A}\psi_{j}\psi_{j}^{-1}=\id_{A},\]
so
\[\sigma_{i}(\id_{A})\widetidle{T}(a_{j})=\sum_{j=1}^{k}\sigma_{i}(\id_{A})\sigma_{i}(\psi_{j})T(\psi_{j}^{-1}a_{j}),\]
hence $\sigma_{i}(\id_{A})\widetilde{T}(a_{j})$ agrees with $T(a_{j})$ on a set of size at least $(1-\varepsilon)d_{i}$ if $i$ is sufficiently large.

	 So, if $W$ is a subspace of $l^{\infty}(\NN,l^{p}(d_{i}))$ which $\varepsilon$-contains $\alpha_{S}(\Hom_{\R,l^{p}}(S,F,r,\delta,\sigma_{i})),$ then $\sigma_{i}(\id_{A})(W)$ $2\varepsilon$-contains $\alpha_{S_{A}}(\Hom_{\R_{A},l^{p}}(S_{A},F',r',\delta',\sigma_{i})).$ This shows that 
\[\dim_{\Sigma,l^{p}}(\id_{A}V,R_{A})\leq \frac{1}{\mu(A)}\dim_{\Sigma,l^{p}}(V,R).\] 

	Note that this implies $\mu(A)\dim_{\Sigma_{A},l^{p}}(V_{A},\R_{A})=\dim_{\Sigma,l^{p}}(V,\R)$ when $\mu(A)$ is rational. If $\mu(A)$ is not rational, let $A_{n}\subseteq A\subseteq B_{n}$ with $A_{n}$ increasing, $B_{n}$ decreasing $\mu(A_{n}),\mu(B_{n})$ are rational and $\mu(A_{n}),\mu(B_{n})\to \mu(A).$ Then by considering compressions
\[\frac{1}{\mu(A_{n})}\dim_{\Sigma,l^{p}}(V,\R)=\dim_{\Sigma_{A_{n}},l^{p}}(V_{A_{n}},\R_{A_{n}})\leq \frac{\mu(A)}{\mu(A_{n})}\dim_{\Sigma_{A},l^{p}}(V_{A},\R_{A})\]
\[\frac{1}{\mu(B_{n})}\dim_{\Sigma,l^{p}}(V,\R)=\dim_{\Sigma_{B_{n}},l^{p}}(V_{B_{n}},\R_{B_{n}})\geq \frac{\mu(B_{n})}{\mu(A)}\dim_{\Sigma_{A},l^{p}}(V_{A},\R_{A}),\]
let $n\to\infty$ to complete the proof.

\end{proof}

	We now show that dimension is subadditive under exact sequences. Unfortunately, we cannot handle superadditivity even in the case of direct sums, not even in the case of Hilbert spaces. Unfortunately, the proof for superadditivity given in \cite{Me}, Theorem 4.7 does not carry over to our setting. The difficulty is in getting a bound analogous to \cite{Me} Lemma 4.3 for our different version of approximate dimension.

\begin{theorem}\label{T:subaddexact} Let $((X,\mu),\R,\Phi,V,\Sigma)$ be a dimensional tuple, and let $W\subseteq V$ be a closed $\R$-invariant subspace. Then  for every $1\leq p<\infty,$ we have the following inequalities:
\[\dim_{\Sigma,l^{p}}(V,\R)\leq \dim_{\Sigma,l^{p}}(V/W,\R)+\dim_{\Sigma,l^{p}}(W,\R),\]
\[\underline{\dim}_{\Sigma,l^{p}}(V,\R)\leq \dim_{\Sigma,l^{p}}(V/W,\R)+\underline{\dim}_{\Sigma,l^{p}}(W,\R),\]
\[\underline{\dim}_{\Sigma,l^{p}}(V,\R)\leq \underline{\dim}_{\Sigma,l^{p}}(V/W,\R)+\dim_{\Sigma,l^{p}}(W,\R).\]
\end{theorem}
\begin{proof} Let $S_{2}=(w_{j})_{j=1}^{\infty}$ be a dynamically generating sequence for $W,$ and $(a_{j})_{j=1}^{\infty}$ a dynamically generating sequence for $V/W.$ Let $v_{j}\in V$ be such that $v_{j}+W=a_{j},$ and $\|v_{j}\|\leq 2\|a_{j}\|.$ Let $S$ be the sequence
\[v_{1},w_{1},v_{2},w_{2},\cdots\]
we shall use $S$ and the pseudonorms
\[\|T\|_{S_{1},i}=\sum_{j=1}^{\infty}\frac{1}{2^{j}}\|T(a_{j})\|,\]
\[\|T\|_{S_{2},i}=\sum_{j=1}^{\infty}\frac{1}{2^{j}}\|T(w_{j})\|,\]
\[\|T\|_{S,i}=\sum_{j=1}^{\infty}\frac{1}{2^{j}}\|T(w_{j})\|+\sum_{j=1}^{\infty}\frac{1}{2^{j}}\|T(v_{j})\|\]
to do our calculation.

	Let $\varepsilon>0,$ and choose $m\in \NN$ such that $2^{-m}<\varepsilon.$ Let $e\in F_{1}\subseteq \Phi$ be finite, $m\leq m_{1}\in \NN$ and $\delta_{1}>0$ to be determined later. By Lemma 3.8 in \cite{Me} Choose $0<\delta<\delta_{1},$ and $F_{1}\subseteq E\subseteq \Phi$ finite and $m_{1}\leq k\in \NN$ so that if $G$ is a Banach space and 
\[T\colon V_{E,2k}\to G\]
has $\|T\|\leq 2,$ and 
\[\|T\big|_{W\cap V_{E,2k}}\|<\delta,\]
then there is a $A\colon (V/W)_{F_{1},m_{1}}\to G$ with $\|A\|\leq 3,$ and 
\[\|A(\psi a_{j})-T(\psi x_{j})\|<\delta_{1},\]
for all $1\leq j,k\leq m_{1}$ and $\psi \in (F_{1}\cup F_{1}^{*}\cup \{e\})^{m_{1}}.$ 

	By finite-dimensionality, we can find a $F'\supseteq E,$ $m'\geq 2k,$  and $0<\delta'<\delta_{1}$ so that if $G$ is a Banach space and $T\colon V_{F',m'}\to G$ has 
\[\|T(\psi x_{j})\|\leq \delta'\|\psi x_{j}\|\]
for all $1\leq j\leq m',\psi\in (F'\cup F'^{*}\cup \{\id\})^{m'}$ then
\[\|T\big|_{W\cap V_{E,2k}}\|<\delta.\]

	Define  $\Xi \colon \Hom_{\R,l^{p}}(S,F',m',\delta',\sigma_{i})\to \Hom_{\R,l^{p}}(S_{2},F',m',\delta',\sigma_{i})$ by
\[\Xi(T)=T\big|_{W_{F',m'}}.\]

	Find
\[\Theta\colon \im(\Xi)\to \Hom_{\R,l^{p}}(S,F',m',\delta',\sigma_{i})\]
so that 
\[\Xi\circ \Theta=\id.\]

	Then
\[(T-\Theta(\Xi(T))(\psi v_{j})=0\]
for all $1\leq j\leq m'$ and $\psi\in (F_{1}\cup F_{1}'\cup \{id\})^{m'}.$ Thus our assumption implies that we can find a $A\colon (V/W)_{F_{1},m_{1}}\to l^{p}(d_{i})$ so that 
\[\|T(\psi x_{j})-A(\psi a_{j})\|<\delta_{1}\]
for all $1\leq j\leq m_{1},\psi \in (F_{1}\cup F_{1}^{*} \cup \{\id\})^{m_{1}},$ with $\|A\|\leq 3.$  

	Thus whenever $\psi \in (F_{1}\cup F_{1}^{*} \cup \{\id\})^{m_{1}},$ and $C\subseteq \{1,\cdots,d_{i}\}$ we have
\[\|A(\psi a_{j})-\sigma_{i}(\psi)A(a_{j})\|_{l^{p}(C)}\leq 2\delta_{1}+\|T(\psi x_{j})-\sigma_{i}(\psi)A(a_{j})\|_{l^{p}(C)},\]
so $A\in \Hom_{\R,l^{p}}(S_{1},F_{1},m_{1},3\delta_{1},\sigma_{i})_{3}.$ The rest is easy.

\end{proof}

\section{Preliminary Results On Direct Integrals}

\begin{definition}\emph{ Let $(X,\mu)$ be a standard measure space, and $V=(V_{x})_{x\in X}$ a family of Banach spaces. We say that $V$ is} measurable \emph{if there are sequences $(v^{(j)}_{x})_{x\in X,j\in \NN},(\phi^{(j)}_{x})_{x\in X,j\in \NN}$ with $v^{(j)}_{x}\in V_{x},\phi^{(j)}_{x}\in V_{x}^{*}$ satisfying the following properties}\end{definition}

\begin{list}{Property \arabic{pcounter}:~}{\usecounter{pcounter}}
\item $x\mapsto\ip{v^{(j)}_{x},\phi^{(k)}_{x}}_{x\in X}$\mbox{ is measurable for all $j,k$}
\item $\overline{\Span}^{\|\cdot\|}\{v^{(j)}_{x}:j\in \NN\}=V_{x}$\mbox{ for almost every $x$}
\item $\overline{\Span}^{\wk^{*}}\{\phi^{(j)}_{x}:j\in \NN\}=V_{x^{*}}$\mbox{ for almost every $x$}
\item $x\mapsto \|\sum_{j}f(j)v^{(j)}_{x}\|$\mbox{ is measurable for all $f\in c_{c}(\NN)$}
\item $x\mapsto \|\sum_{j}f(j)\phi^{(j)}_{x}\|$\mbox{ is measurable for all $f\in c_{c}(\NN)$}
\end{list}

	It is a fact that if we are given properties $1-3,$ then property $4$ is actually equivalent to property $5.$ 

	We define the set of measurable vector fields, $\Meas(X,V)$ , to be all fields $(v_{x})_{x\in X}$ of vectors in $X$ such that  $v_{x}\in V_{x}$ for all $x$ and $x\mapsto \ip{v_{x},\phi^{(j)}_{x}}$ is measurable for all $j\in \NN.$ Note that our axioms imply that 
\[\|v_{x}\|=\sup_{\substack{f\in c_{c}(\NN,\QQ[i]),\\ \left\|\sum_{j}f(j)\phi^{(j)}_{x}\right\|<1}}\left|\sum_{j}\ip{v_{x},\phi^{(j)}_{x}}\right|.\]

so that the norm of a measurable vector fields is a measurable function. We also define $\Meas(X,V^{*})$ to be all fields of vectors $(\phi_{x})_{x}$ such that $\phi_{x}\in V_{x}^{*}$ for all $x\in X$ and $x\mapsto \ip{v^{(j)}_{x},\phi_{x}}$ is measurable for all $j\in \NN.$ As above $\|\phi_{x}\|$ is measurable. We leave it as an exercise to verify that if $v\in \Meas(X,V),\phi\in \Meas(X,V^{*})$ then $x\mapsto \ip{v_{x},\phi_{x}}$ is measurable.

	For $1\leq p<\infty,$ we define the $L^{p}$-direct integral of $V$ denoted
\[\int_{X}^{\oplus_{p}}V_{x}\,d\mu(x)\]
to be all $v\in \Meas(X,V)$ so that 
\[\|v\|_{p}^{p}=\int_{X}\|v_{x}\|^{p}\,d\mu(x)<\infty.\]

	H\"{o}lder's inequality shows that $\int_{X}^{\oplus_{p}}V_{x}\,d\mu(x)$ is a vector space.

\begin{proposition} Let $(X,\mu)$ be a standard measure space and $V$ a measurable field of Banach spaces over $X.$ Then for $1\leq p<\infty,$
\[\int_{X}^{\oplus_{p}}V_{x}\,d\mu(x)\]
is a separable Banach space. Further a sequence $(w^{(j)})_{j=1}^{\infty}$ in $\int_{X}^{\oplus_{p}}V_{x}\,d\mu(x)$ has
\[\Span\{\chi_{A}w^{(j)}:A\mbox{ measurable} ,j\in \NN\}\]
dense in $\int_{X}^{\oplus_{p}}V_{x}\,d\mu(x)$ if and only if for almost every $x,$ $(w^{(j)}_{x})_{j=1}^{\infty}$ spans a dense subspace.\end{proposition}

\begin{proof} Let $v^{(j)}_{x},\phi^{(j)}_{x}$ be as in the definition of measurable vector field. We first prove completeness.

	Suppose that $w^{(n)}$ in $\int_{X}^{\oplus_{p}}V_{x}\,d\mu(x)$ has
\[\sum_{n=1}^{\infty}\|w^{(n)}\|_{p}<\infty.\]

	Then,
\begin{align*}
\int_{X}\sum_{n=1}^{\infty}\|w^{(n)}_{x}\|^{p}\,d\mu(x)&\leq \liminf_{N\to \infty}\int_{X}\sum_{n=1}^{N}\|w^{(n)}_{x}\|^{p}\,d\mu(x)\\
&\leq \left(\sum_{n=1}^{N}\|w^{(n)}\|_{p}\right)^{p}\\
&\leq \left(\sum_{n=1}^{\infty}\|w^{(n)}\|_{p}\right)^{p}\\
&<\infty.
\end{align*}

	So for almost every $x,$ $w_{x}=\sum_{n=1}^{\infty}w^{(n)}_{x}$ is norm convergent in $V_{x}.$ It is easy to see by taking limits that $w\in \Meas(X,V).$ By the same inequalities as above we also see that 

\[\left\|w-\sum_{n=1}^{N}w^{(n)}\right\|_{p}\leq \sum_{n=N+1}^{\infty}\|w^{(n)}\|_{p}\to 0,\]
as $N\to \infty,$ and this proves completeness.

	For the second fact, first suppose that $w^{(j)}$ in $\int_{X}^{\oplus_{p}}V_{x}\,d\mu(x)$ is such that $\Span\{w^{(j)}_{x}:j\in \NN\}$ is dense in $V_{x}$ for almost every $x\in X>$ Let $v\in \int_{X}^{\oplus_{p}}V_{x}\,d\mu(x)$ and $\varepsilon>0.$ then up to sets of measure zero,
\[X=\bigcup_{f\in c_{c}(\NN,\QQ[i])}\left\{x\in X:\left\|\sum_{j=1}^{\infty}f(j)w^{(j)}_{x}-v_{x}\right\|<\varepsilon\right\}.\]

	Thus by the usual arguments we can find a measurable $f\colon X\to c_{c}(\NN,\QQ[i])$ such that for almost every $x\in X,$ we have
\[\left\|\sum_{j=1}^{\infty}f(x)(j)w^{(j)}_{x}-v_{x}\right\|<\varepsilon.\]

	Let $F_{n}$ be finite subsets of $\QQ[i]$ which increase to $\QQ[i],$ and so that $0\in F_{n}$ for all $n.$  For $n\in \NN,$ set
\[X_{n}=\{x\in X:f(x)(j)=0\mbox{ for $j\geq n$},f(x)(j)\in F_{n} \mbox{ for all $j$}\}.\]

	If $n$ is sufficiently large then,
\[\int_{X_{n}^{c}}\|v_{x}\|^{p}\,d\mu(x)<\varepsilon.\]

	Thus
\[\int_{X}\left\|\sum_{j=1}^{\infty}\chi_{X_{n}}f(x)(j)w^{(j)}_{x}-v_{x}\right\|^{p}\,d\mu(x)<\varepsilon^{p}+\varepsilon,\]
and it is easy to see that 
\[\sum_{j=1}^{\infty}\chi_{X_{n}}f(x)(j)w^{(j)}_{x}\]
is a finite linear combination of elements of the form $\chi_{A}w^{(j)}_{x}.$ This proves one implication.

	Conversely, suppose that $\chi_{A}w^{(j)}_{x}$ densely span $\int_{X}^{\oplus_{p}}V_{x}\,d\mu(x),$ but that 
\[A=\{x\in X:w^{(j)}_{x}\mbox{ does not densely span $V_{x}$}\}\]
has positive measure. Then there is a $v\in \Meas(A,V)$ so that 
\[d\left(v_{x},\Span\{w^{(j)}_{x}\}\right)\geq 1\]
for all $x\in A.$ But we can find  $\lambda_{1},\cdots,\lambda_{k}\in \CC,$ $j_{1},\cdots,j_{k}\in \NN$ and sets $A_{1},\cdots,A_{k}$ so that 
\[\left\|v-\sum_{j=1}^{k}\lambda_{j}\chi_{A_{j}}w^{(j)}\right\|_{p}<1/2.\]

	Replacing $A_{j}$ with $A\cap A_{j}$ we may assume $A_{j}\subseteq A.$ This clear implies that there is some $x\in A$ so that 
\[\left\|v-\sum_{j=1}^{k}\lambda_{j}\chi_{A_{j}}(x)w^{(j)}_{x}\right\|_{p}<1/2,\]

	and this is a contradiction.

\end{proof}

	Direct integrals arise naturally in the context of representations of equivalence relations.

\begin{definition}\emph{ Let $(\R,X,\mu)$ be a discrete measure preserving equivalence relation, and let $x\to V_{x}$ be measurable field of Banach spaces over $X.$} A representation $\pi$ of $\R$ on $V$ \emph{consists of bounded linear maps $\pi(x,y)\colon V_{y}\to V_{x}$ so that $\pi(z,x)\pi(x,y)=\pi(z,y)$ for $x\thicksim y\thicksim z,$ $\pi(x,x)=\id,$ and for each $v\in \Meas(X,V),\phi\in \Meas(X,V^{*})$ we have that $(x,y)\to \ip{\pi(x,y)v_{y},\phi_{x}}$ is a measurable map $\R\to \CC.$ We say that $\pi$ is} uniformly bounded \emph{if there is a $C>0$ so that $\|\pi(x,y)v\|\leq C$ for all $(x,y)\in \R,$ $v\in V_{y}.$ }\end{definition}

	Note that if $\pi$ is uniformly bounded, then for every $1\leq p<\infty,$ we get a uniformly bounded action of $\R$ on $\int_{X}^{\oplus_{p}}V_{x}\,d\mu(x)$ by
\[(\phi\cdot v)_{x}=\chi_{\ran(\phi)}(x)\pi(x,\phi^{-1}(x))v_{\phi^{-1}(x)}.\]

	Our work in this section has the following corollary which will allow us to work fiberwise in the case of representations on measurable fields. This will be used quite heavily in Section \ref{S:cohom}.

\begin{cor} Let $(\R,X,\mu)$ be a discrete measure-preserving equivalence relation, with a representation $\pi$ on a measurable field of Banach spaces $x\to V_{x}.$  If $w^{(j)}\in \int_{X}^{\oplus_{p}}V_{x}\,d\mu(x)$ is bounded, then $w^{(j)}$ is dynamically generating if and only if for almost every $x,$
\[\overline{\Span\{\pi(x,y)w^{(j)}_{y}:y\thicksim x\}}^{\|\cdot\|_{V_{x}}}=V_{x}.\]
\end{cor}

\section{Computations for $L^{p}(\R,\overline{\mu}).$}

	Here we prove that 
\[\dim_{\Sigma,l^{p}}(L^{p}(\R,\overline{\mu})^{\oplus n},\R)=\underline{\dim}_{\Sigma,l^{p}}(L^{p}(\R,\overline{\mu})^{\oplus n},\R)=n.\]

	We must take a different approach than that in \cite{Me}, as the operators defined there will not fill up enough space if we use our different version of $\varepsilon$-dimension. Instead, we shall take a more probabilistic approach.
\begin{proposition} Fix $1\leq p<\infty.$ Let $A_{i}\subseteq B(l^{p}(n,\nu_{n})),$ be measurable, where $\nu_{n}$ is the uniform measure,  and suppose  that 
\[\limsup_{i\to \infty}\left(\frac{\vol(A_{i})}{\vol(B(l^{p}(n,\nu_{i}))}\right)^{1/2n}\geq \alpha.\]

	Then there is a $\kappa(\alpha,\varepsilon,p)$ with
\[\lim_{\varepsilon\to 0}\kappa(\alpha,\varepsilon,p)=1,\]
so that 
\[\liminf_{i\to \infty}\frac{1}{n}d_{\varepsilon}(A_{i},\|\cdot\|_{p})\geq \kappa(\alpha,\varepsilon,p).\]

\end{proposition}

\begin{proof} If the claim is false, then there is a $\kappa<1,$ so that for every $\varepsilon>0,$
\[\kappa>\liminf_{i\to \infty}\frac{1}{d_{i}}d_{\varepsilon}(A_{i},\|\cdot\|_{p}),\]

	Then for all large $n,$ we can find a subspace $W\subseteq l^{p}(n)$ with $\dim(W)\leq \kappa n,$ and $A\subseteq_{\varepsilon} W.$ This implies that 
\[A\subseteq \bigcup_{\substack{B\subseteq\{1,\cdots,n\},\\ |B|\leq \varepsilon n}}((1+\varepsilon)\Ball(\chi_{B^{c}}(W))+\varepsilon\Ball(l^{p}(B^{c},\nu_{B^{c}}))\times \Ball(l^{p}(B,\nu_{B})).\]

	Since $\chi_{B^{c}}(W)$ has dimension at most $\kappa n,$ we can find a $\varepsilon$-dense subset of $(1+\varepsilon)\Ball(\chi_{B^{c}}(W))$ of cardinality at most $\left(\frac{2+4\varepsilon}{\varepsilon}\right)^{2\kappa n}.$ Thus
\begin{align*}
&\vol((1+\varepsilon)\Ball(\chi_{B^{c}}(W))+\varepsilon\Ball(l^{p}(B^{c},\nu_{B^{c}}))\leq\\
 &\left(\frac{2+4\varepsilon}{\varepsilon}\right)^{2\kappa n}\vol(\Ball(l^{p}(B^{c},\nu_{B^{c}}))(2\varepsilon)^{2|B^{c}|}.
\end{align*}

	So $\frac{\vol(A)}{\vol(\Ball(l^{p}(n,\nu_{n}))}$ is at most
\[\sum_{\substack{B\subseteq\{1,\cdots,d_{i}\}\\ |B|\leq \varepsilon n}}2^{|B^{c}|}(\varepsilon)^{2(|B^{c}|-\kappa n)}(2+4\varepsilon)^{2\kappa n}\frac{\vol(\Ball(l^{p}(B^{c},\nu_{B^{c}}))\vol(\Ball(l^{p}(B,\nu_{B}))}{\vol(\Ball(l^{p}(n,\nu_{n}))}.\]

	We have that the above sum is
\[\sum_{r=0}^{\lfloor{\varepsilon n\rfloor}}2^{n-r}(\varepsilon)^{2(n(1-\kappa)-r)}(2+4\varepsilon)^{2\kappa n}\binom{n}{r}V(r,n,p)\]
where
\[V(r,n,p)=\frac{r^{2r/p}(n-r)^{2(n-r)/p}\Gamma(1+\frac{2n}{p})}{\Gamma(1+\frac{2r}{p})\Gamma(1+\frac{2n-2r}{p})n^{2n/p}}.\]

	By Stirling's Formula we see that 
\[V(r,n,p)\leq C(p),\]
where $C(p)$ is a constant which depends only on $p.$ 

	Further if $n$ is sufficiently large and $\varepsilon<1/2,$ then by Stirling's Formula
\[\binom{n}{r}\leq \binom{n}{\lfloor{\varepsilon n\rfloor}}\leq C\left(\frac{n}{\lfloor{\varepsilon n\rfloor}}\right)^{\lfloor{\varepsilon n\rfloor}}\left(\frac{n}{n-\lfloor{\varepsilon n\rfloor}}\right)^{n-\lfloor{\varepsilon n\rfloor}},\]
for some constant $C.$ 

	Putting this altogether, we have that 
\[\alpha\leq \sqrt{2}\varepsilon^{(1-\kappa)-\varepsilon}(2+4\varepsilon)^{\kappa}\left(\frac{1}{\varepsilon}\right)^{\varepsilon}\left(\frac{1}{1-\varepsilon}\right)^{1-\varepsilon}.\]
	
	Since $\kappa<1,$ the right-hand side tends to zero as $\varepsilon\to 0$ so we have a contradiction.

\end{proof}

\begin{theorem}\label{T:LpDim}  Let $\R$ be a sofic discrete  measure-preserving equivalence relation on a standard probability space $(X,\mu).$ For all $1\leq p\leq 2,$ we have
\[\dim_{\Sigma,l^{p}}(L^{p}(\R,\overline{\mu})^{\oplus n},\R)=\underline{\dim}_{\Sigma,l^{p}}(L^{p}(\R,\overline{\mu})^{\oplus n},\R)=n.\]
\end{theorem}

\begin{proof}  We shall present the proof when $n=1.$ Since our approach is probabilistic, it is not hard to generalize the proof for general $n.$

Let $\Sigma$ be a sofic approximation of $\R,$ and let $\id\in \Phi=\Phi_{0}\cup \mathcal{P},$ where $\Phi_{0}$ is a graphing of $\R,$ and $\mathcal{P}$ is generating family of projections in $L^{\infty}(X,\mu).$ Let $\id \in F\subseteq \Phi$ be finite, $m\in \NN,\delta>0.$ We use $S=(\chi_{\Delta})$ to do our computation. It is clear that 
\[\dim_{\Sigma,l^{p}}(L^{p}(R,\overline{\mu}),\R)\leq 1,\]
so it suffices to show that 
\[\underline{\dim}_{\Sigma,l^{p}}(L^{p}(R,\overline{\mu}),\R)\geq 1.\]

	Let
\[C=W^{*}(\{v_{\psi}pv_{\psi}^{*}:\psi\in F^{m},p\in \mathcal{P}\cap F^{m}\}),\]
 and let $\chi_{B_{1}},\cdots ,\chi_{B_{r}}$ be the minimal projections in $C.$ Let $\{A_{1},\cdots,A_{q}\}$ be a partition refining $\{B_{1},\cdots,B_{r}\},$ which we will assume to be sufficiently fine in a manner to be determined later. We may assume that $\Sigma$ is eventually a homomorphism on $W^{*}(\{A_{j}\}_{j=1}^{q}),$ there are $E_{j}\subseteq [[\R]],$ \
\[\mathcal{O}_{A_{j}}\colon =\{(x,y)\in R:x\in A_{j}\}=\bigsqcup_{\psi\in E_{j}}\graph(\psi),\]
and that
\[F^{m}\subseteq E_{1}^{-1}+E_{2}^{-1}+\cdots +E_{q}^{-1}.\]

	We may also assume that for every $\psi\in E_{j}$ and for all large $i,$ we have that $\dom(\sigma_{i}(\psi))\subseteq \sigma_{i}(A_{j}).$

	Note that if $f\in L^{p}(\R,\overline{\mu}),$ then we can uniquely write
\[\id_{A_{j}}f=\sum_{\psi \in E_{j}}f_{\psi}\chi_{\graph(\psi)},\]
where $f_{\psi}\in L^{p}(\dom(\psi),\mu)$ and the sum converges in $\|\cdot\|_{p}.$
	Fix $\eta>0,$ and let $F_{j}\subseteq E_{j}$ be finite and so that for all $\psi\in F^{m},$
\[\dist_{\|\cdot\|_{2}}(\psi,F^{m})<\eta.\]

	For $\xi\in l^{p}(d_{i},\nu_{i})$ define 
\[T^{(j)}_{\xi}(f)=\sum_{\psi \in F_{j}}\EE_{\dom(\psi)}(f_{\psi})\sigma_{i}(\psi^{-1})\xi,\]
where for a measurable $A\subseteq X,$ and $f\in L^{1}(A,\mu)$ we use 
\[\EE_{A}(f)=\frac{1}{\mu(A)}\int_{A}f\,d\mu.\]

	Finally set
\[T_{\xi}=\sum_{j=1}^{q}T^{(j)}_{\xi}(f).\]

	We claim that if $\{A_{1},\cdots,A_{q}\}$ is sufficiently fine, and $i$ is sufficiently large, then 
\begin{equation}\label{E:probarg}
\frac{\mu(\{\xi\in \Ball(l^{2}(d_{i},\nu_{i})):\|T_{\xi}\|_{L^{p}\to l^{p}}\leq 1,\mbox{ for all $1\leq p\leq 2$}\})}{\mu(\Ball(l^{2}(d_{i},\nu_{i})}\to 1.
\end{equation}

	By interpolation it suffices to show that
\[\frac{\mu(\{\xi\in \Ball(l^{2}(d_{i},\nu_{i})):\|T_{\xi}\|_{L^{1}\to l^{1}}\leq 1,)}{\mu(\Ball(l^{2}(d_{i},\nu_{i})}\to 1,\]
\[\frac{\mu(\{\xi\in \Ball(l^{2}(d_{i},\nu_{i})):\|T_{\xi}\|_{L^{2}\to l^{2}}\leq 1,)}{\mu(\Ball(l^{2}(d_{i},\nu_{i})}\to 1,\]

	Let us first do the $l^{2}$ case. We have that
\[\|T^{(j)}_{\xi}(f)\|_{2}^{2}\leq \sum_{\psi\in F_{j}}|\EE_{\dom(\psi)}(f_{\psi})|^{2}\|\sigma_{i}(\psi)^{-1}\xi\|_{2}^{2}+\]
\[\sum_{\phi \ne \psi \in F_{j}}|\EE_{\dom(\psi)}(f_{\psi})\EE_{(\dom(\phi)}(f_{\phi})||\ip{\sigma_{i}(\psi)^{-1}\xi,\sigma_{i}(\phi)^{-1}\xi}|\leq\]
\[\sum_{\psi\in F_{j}}\frac{\|f_{\psi}\|_{2}^{2}}{\mu(\dom(\psi)}\|\sigma_{i}(\psi)^{-1}\xi\|_{2}^{2}+\]
\[\sum_{\phi \ne \psi \in F_{j}}\frac{\|f_{\psi}\|_{2}\|f_{\phi}\|_{2}}{\mu(\dom(\psi)^{1/2}\mu(\dom(\phi)^{1/2}}|\ip{\sigma_{i}(\psi)^{-1}\xi,\sigma_{i}(\phi)^{-1}\xi}|.\]

	Since
\[\frac{1}{\vol(\Ball(l^{2}(d_{i},\nu_{i})}\int_{\Ball(l^{2}(d_{i},\nu_{i})}\|\sigma_{i}(\psi)^{-1}\xi\|_{2}^{2}\,d\xi\leq \frac{|\dom(\sigma_{i}(\psi)^{-1}|}{d_{i}}\to \mu(\dom(\psi)),\]
\[\frac{1}{\vol(\Ball(l^{2}(d_{i},\nu_{i})}\int_{\Ball(l^{2}(d_{i},\nu_{i})}\ip{\sigma_{i}(\psi)^{-1}\xi,\sigma_{i}(\phi)^{-1}\xi}\,d\xi=\frac{2n}{2n+2}\tr(\sigma_{i}(\phi)\sigma_{i}(\psi))\to 0,\]
it follows by concentration of measure that $\PP(\|T^{(j)}_{\xi}\|\leq 2\mbox{ for all $j$})\to 1.$ If $\|T^{(j)}(\xi)\|_{2}\leq 2$ for all $j,$ then
\[\|T(f)\|_{2}^{2}=\sum_{j=1}^{q}\|T^{(j)}_{\xi}(\id_{A_{j}}f)\|_{2}^{2}\leq 4\sum_{j=1}^{q}\|\id_{A_{j}}f\|_{2}^{2}\leq 4\|f\|_{2}^{2}.\]

	For the $l^{1}$-case, simply note that 
\[\|T^{(j)}_{\xi}(f)\|_{1}\leq \sum_{\psi\in F_{j}}\frac{\|f_{\psi}\|_{1}}{\mu(\dom(\psi))}\|\sigma_{i}(\psi)^{-1}\xi\|.\]

	Since
\[\frac{1}{\vol(\Ball(l^{2}(d_{i},\nu_{i})}\int_{\Ball(l^{2}(d_{i},\nu_{i})}\|\sigma_{i}(\psi)^{-1}\xi\|_{1}\,d\xi=\]
\[\frac{|\dom(\sigma_{i}(\psi)^{-1})|}{d_{i}}\frac{1}{\vol(\Ball(l^{2}(d_{i},\nu_{i})}\int_{\Ball(l^{2}(d_{i},\nu_{i})}|\xi_{1}|d\xi\leq\]
\[\frac{|\dom(\sigma_{i}(\psi)^{-1})|}{d_{i}}\left(\frac{1}{\vol(\Ball(l^{2}(d_{i},\nu_{i})}\int_{\Ball(l^{2}(d_{i},\nu_{i})}|\xi_{1}|^{2}\,d\xi\right)^{1/2}=\]
\[\frac{|\dom(\sigma_{i}(\psi)^{-1})|}{d_{i}}\left(\frac{1}{\vol(\Ball(l^{2}(d_{i},\nu_{i})}\int_{\Ball(l^{2}(d_{i},\nu_{i})}\|\xi\|_{2}^{2}\,d\xi\right)^{1/2}\leq \]
\[\frac{|\dom(\sigma_{i}(\psi)^{-1})|}{d_{i}}\to \mu(\dom(\psi)).\]

	So, again by concentration of measure
\[\PP(\{\xi:\|T^{(j)}_{\xi}\|_{L^{1}\to l^{1}}\leq 2\mbox{ for all $j$})\to 1.\]

	If $\|T^{(j)}_{\xi}\|_{L^{1}\to l^{1}}\leq 2$ for all $j,$ it is again easy to see that $\|T_{\xi}\|_{L^{1}\to l^{1}}\leq 1.$ Thus $(\ref{E:probarg})$ holds.  Suppose $\phi\in F^{m},$ by our choice of $E_{1},\cdots,E_{q},$ we may write
\[\phi=\sum_{j=1}^{q}\sum_{\psi\in E_{j}}c_{j,\psi}\psi^{-1},\]
where $c_{j,\psi}\in \{0,1\},$ further
\[\left\|\phi-\sum_{j=1}^{q}\sum_{\psi\in F_{j}}c_{j,\psi}\psi^{-1}\right\|_{2}^{2}<\eta^{2}.\]

	So
\[\|T(\chi_{\graph(\phi)})-\sigma_{i}(\phi)T(\chi_{\Delta})\|_{2}=\left\|\left(\sum_{j=1}^{q}\sum_{\psi\in F_{j}}c_{j,\psi}\sigma_{i}(\psi)^{-1}-\sigma_{i}(\phi)\right)\xi\right\|_{2},\]
and for most $\xi$ this is at most $2\eta,$ again by concentration of measure. Thus we have shown that 
\[\frac{\vol(\alpha_{S}(\Hom_{\R,l^{p}}(S,F,m,\delta,\sigma_{i}))}{\vol(\Ball(l^{2}(d_{i},\nu_{i}))}\to 1,\]
and since
\[\inf_{i}\left(\frac{\vol(\Ball(l^{2}(d_{i},\nu_{i}))}{\vol(\Ball(l^{p}(d_{i},\nu_{i}))}\right)^{1/2d_{i}}>0,\]
we are done by the proceeding Lemma.

\end{proof}

	We can prove a nice fact in the case of the action of $\R$ on $L^{2}(\R,\overline{\mu})$ but we will need a generalization of our previous volume packing Lemma.

\begin{proposition}\label{P:volpackproj}    There is a function $\kappa(\alpha,\varepsilon)$ with 
\[\lim_{\varepsilon\to 0}\kappa(\alpha,\varepsilon)=1\]
for all $\alpha$ which has the following property.  Let $d_{i}$ be a sequence of integers going to infinity, and let $A_{i}\subseteq \Ball(l^{2}(d_{i})),$ and let $p_{i}$ be a projection on $l^{2}(d_{i}),$ so that $\tr(p_{i})$ converges. If
\[\liminf_{i\to \infty}\left(\frac{\vol(A_{i})}{\vol(\Ball(l^{2}(d_{i})))}\right)^{1/d_{i}}\geq \alpha,\]
then 
\[\liminf_{i\to \infty}\frac{1}{d_{i}\tr(p_{i})}d_{\varepsilon}(p_{i}A_{i},\|\cdot\|_{2})\geq \kappa(\alpha,\varepsilon).\]
\end{proposition}

\begin{proof} If the claim is false, then we can find $\kappa<1,$ sets $A_{i}$ as in the proposition, and subspaces $V_{i}\subseteq l^{2}(d_{i})$ with $\dim(V_{i})\leq \kappa \tr(p_{i})d_{i},$ so that $p_{i}A_{i}\subseteq_{\varepsilon}V_{i}.$ This implies that 
\[p_{i}A_{i}\subseteq \bigcup_{\substack{ B\subseteq \{1,\cdots,d_{i}\},\\|B|\leq \varepsilon d_{i}}}[(1+\varepsilon)\Ball_{\|\cdot\|_{2}}(\chi_{B^{c}}(V_{i})+\varepsilon\Ball(l^{2}(B^{c}))]\times \Ball(p_{i}l^{2}(B)).\]

	Let $q_{i}=\tr(p_{i}),q=\lim q_{i},$ also let $V(k)$ be the volume of the $k$-dimensional ball in $l^{2}(d_{i}).$ Then we have
\[\vol(p_{i}A_{i})\leq \vol[(1+\varepsilon)\Ball(p_{i}\chi_{B^{c}}(V_{i})+\varepsilon\Ball(l^{2}(B^{c}))]V(\dim(p_{i}l^{2}(B))).\]

	Let $S_{B}$ be a maximal $\varepsilon$-separated subset of $(1+\varepsilon)\Ball(p_{i}\chi_{B^{c}}(V_{i})),$ then
\[|S_{B}|\leq \left(\frac{2+2\varepsilon}{\varepsilon}\right)^{\dim (p_{i}\chi_{B^{c}}(V_{i}))}\leq \left(\frac{2+2\varepsilon}{\varepsilon}\right)^{\kappa q_{i}d_{i}}.\]

	Thus
\begin{align*}
\vol(p_{i}A_{i})&\leq \sum_{\substack{ B\subseteq \{1,\cdots,d_{i}\},\\|B|\leq \varepsilon d_{i}}}\left(\frac{2+2\varepsilon}{\varepsilon}\right)^{\kappa q_{i}d_{i}}(2\varepsilon)^{\dim(p_{i}l^{2}(B^{c}))}V(\dim(p_{i}l^{2}(B)))V(\dim(p_{i}l^{2}(B^{c}))\\
&\leq \sum_{\substack{ B\subseteq \{1,\cdots,d_{i}\},\\|B|\leq \varepsilon d_{i}}}4^{\kappa q_{i}d_{i}}2^{d_{i}q_{i}}\varepsilon^{d_{i}(1-\kappa)q_{i}}V(\dim(p_{i}l^{2}(B))V(\dim(p_{i}l^{2}(B^{c})).
\end{align*}

Thus
\[\frac{\vol(p_{i}A_{i})}{V(q_{i}d_{i})}\leq \sum_{\substack{ B\subseteq \{1,\cdots,d_{i}\},\\|B|\leq \varepsilon d_{i}}}4^{\kappa q_{i}d_{i}}2^{d_{i}q_{i}}\varepsilon^{d_{i}(1-\kappa)q_{i}}\frac{V(\dim(p_{i}l^{2}(B))V(\dim(p_{i}l^{2}(B^{c}))}{V(q_{i}d_{i})}.\]

	Now 
\[V(k)=\frac{\pi^{k}}{k!},\]
so by Striling's Formula there is a constant $C>0$ so that 
\[\frac{V(\dim(p_{i}l^{2}(B)))V(\dim(p_{i}l^{2}(B^{c}))}{V(q_{i}d_{i})}\leq C\pi^{\varepsilon d_{i}}\frac{(q_{i}d_{i})^{q_{i}d_{i}}e^{\varepsilon d_{i}}\sqrt{2\pi q_{i}d_{i}} }{(q_{i}-\varepsilon)^{(q_{i}-\varepsilon)d_{i}}\sqrt{2\pi(q_{i}-\varepsilon)d_{i}}}.\]

	Thus

\[\limsup_{i\to \infty}\left(\frac{\vol(p_{i}A_{i})}{V(q_{i}d_{i})}\right)^{1/d_{i}}\leq \frac{q^{q}}{(q-\varepsilon)^{q-\varepsilon}\varepsilon^{\varepsilon}(1-\varepsilon)^{(1-\varepsilon)}}4^{\kappa q}2^{q}\varepsilon^{(1-\kappa)q}.\]

Since $\vol(A_{i})\leq \vol(p_{i}A)V((1-q_{i})d_{i})$ we have
\begin{align*}
\alpha&\leq \frac{q^{q}}{(q-\varepsilon)^{q-\varepsilon}\varepsilon^{\varepsilon}(1-\varepsilon)^{(1-\varepsilon)}}4^{\kappa q}2^{q}\varepsilon^{(1-\kappa)q}\times\\
&\limsup_{i\to \infty}\left(\frac{\vol(q_{i}d_{i})V((1-q_{i})d_{i})}{V(d_{i})}\right)^{1/d_{i}}\\
&=\frac{(1-q)^{1-q}}{(q-\varepsilon)^{q-\varepsilon}\varepsilon^{\varepsilon}(1-\varepsilon)^{(1-\varepsilon)}}4^{\kappa q}2^{q}\varepsilon^{(1-\kappa)q}.
\end{align*}

Letting $\varepsilon\to 0,$ we obtain a contradiction.

\end{proof}

\begin{theorem}\label{T:HSlowerbound}  Let $\R$ be a discrete-measure preserving sofic equivalence relation on $(X,\mu).$ Let $\mathcal{H}$ be a separable unitary representation of $\R$ such that the action of $\R$ on $\mathcal{H}$ extends to the von Neumann algebra $L(\R).$ For any sofic approximation $\Sigma$ of $\R,$ we have
\[\dim_{\Sigma,l^{2}}(\mathcal{H},\R)=\underline{\dim}_{\Sigma,l^{2}}(\mathcal{H},\R)=\dim_{L(\R)}(\mathcal{H}).\]
\end{theorem}
\begin{proof} We first show that 
\[\underline{\dim}_{\Sigma,l^{2}}(\mathcal{H},\R)\geq \dim_{L(\R)}(\mathcal{H}).\]

Our hypothesis implies that as a representation of $\R,$
\[\mathcal{H}\cong \bigoplus_{j=1}^{\infty}L^{2}(\R,\overline{\mu})q_{j},\]
with $q_{j}\in \Proj(L(\R)).$ 

	As in Theorem \ref{T:LpDim} we shall deal with the case that $\mathcal{H}=L^{2}(\R,\mu)q$ for some $q\in \Proj(L(\R)),$ it is easy to see that our proof generalizes.

	Thus $\mathcal{H}$ is unitarily equivalent to a subrepresentation of $L^{2}(R,\overline{\mu})$ so we may as well assume that it is a subrepresentation of $L^{2}(R,\overline{\mu})$. Let $p$ be the projection onto $\mathcal{H},$ we use $\widehat{p}=p\chi_{\Delta}$ to do our calculation. Fix a graphing $\Phi$ of $\R,$ and 
\[\sigma_{i}\colon *-\Alg(\Phi)\to M_{d_{i}}(\CC),\]
a sofic approximation. Let $A=*-\Alg(\Phi,p),$ we may find an extension to a sofic approximation
\[\widetilde{\sigma_{i}}\colon A\to M_{d_{i}}(\CC),\]
by perturbing elements slightly, we may assume that $p_{i}=\widetilde{\sigma_{i}}(p)$ is a projection for all $i.$ Let $T_{\xi}$ be the operator constructed in the proof of Theorem $\ref{T:LpDim}.$ Fix $F\subseteq \Phi$ finite, $m\in \NN,\delta>0.$ 

 We know that for every $F'\subseteq \Phi$ finite, $m'\in \NN,$ $\delta'>0$ that  
\[\frac{\vol(\{\xi\in \Ball(l^{2}(d_{i})):T_{\xi}\in \Hom_{\R,l^{2}}(\{\chi_{\Delta}\},F,m,\delta,\sigma_{i})\})}{\vol(\Ball(l^{2}(d_{i}))}\to 1,\]
and that $T_{\xi}(\chi_{\Delta})$ is close to $\xi.$ 

	It is easy to see that if $F',m',\delta'$ are chosen wisely then
\[\Hom_{\R,l^{2}}(\{\chi_{\Delta}\},F',m',\delta',\sigma_{i}))\big|_{pL^{2}(R,\overline{\mu})}\subseteq \Hom_{\R,l^{2}}(\{\widehat{p}\},F,m,\delta,\sigma_{i}),\]
and that $T_{\xi}(\widehat{p})$ is close to $p_{i}.$ Thus the preceding proposition proves the lower bound.

	For the upper bound, let $S=(\chi_{\Delta}q_{j})_{j=1}^{\infty}.$ Fix $\varepsilon>0,$ $m\in \NN,$ it is easy to see that if $F$ is large, and $\delta>0$ is small enough then
\[\{(T(\chi_{\Delta}q_{1}),\cdots,T(\chi_{\Delta}q_{m})):T\in \Hom_{\R,l^{p}}(S,F,m,\delta,\sigma_{i})\}\subseteq_{\varepsilon}\bigoplus_{j=1}^{m}\sigma_{i}(q_{j})\Ball(l^{2}(d_{i})),\]
as 
\[\tr(\sigma_{i}(q_{j})))\to \tau(q_{j}),\]
the desired upperbound is proved.

\end{proof}

	We close this section with a complete computation in the case of direct integrals of finite-dimensional representations. 

\begin{proposition} Let $(R,X,\mu)$ be a discrete, measure-preserving equivalnce relation. Suppose that for some $n\in \NN,$  $|\mathcal{O}_{x}|=n$ for almost for every $x\in X.$ Let $V_{x}$ be a measure-field of finite dimensional vector spaces and $\pi$ a representation of $\R$ on $V_{x}.$ Then for all $1\leq p<\infty,$ and for every sofic approximation $\Sigma$ of $\R,$
\[\dim_{\Sigma,l^{p}}\left(\int_{X}^{\oplus_{p}}V_{x}\,d\mu(x),\R\right)=\underline{\dim}_{\Sigma,l^{p}}\left(\int_{X}^{\oplus_{p}}V_{x}\,d\mu(x),\R\right)=\frac{1}{n}\int_{X}\dim(V_{x})\,d\mu(x).\]

\end{proposition}

\begin{proof} We shall only handle the case when $\dim(V_{x})$ is almost surely constant, say equal to $k.$   The general case will follow by more or less the same proof. Without loss of generality $V_{x}=\CC^{k}$ with the Euclidean norm and $\pi(x,y)$ is a unitary for almost every $(x,y)\in \R.$ Let $\alpha\in [\R]$ be $n$-peroidic and so that up to sets of measure zero, $\R=\{(x,\alpha^{j}(x)):0\leq j\leq n-1\}.$
	Let 
\[b(\alpha^{j}x)=\pi(x,\alpha^{j}x),x\in A,0\leq j\leq n-1.\]

	Then
\[b(\alpha^{j}x)b(\alpha^{k}x)^{-1}=\pi(\alpha^{j}x,\alpha^{k}x),\]
that is 
\[b(y)b(x)^{-1}=\pi(y,x)\]
for $x,y\in \R.$ 

	Define $T\colon L^{p}(X,\mu,\CC^{k})\to L^{p}(X,\mu,\CC^{k})$
by
\[(Tf)(x)=b(x)f(x).\]

	For $\phi \in [[\R]],$ we have
\[\phi\cdot (Tf)(x)=\chi_{\ran(\phi)}(x)\pi(x,\phi^{-1}x)b(\phi^{-1}x)f(\phi^{-1}x)=\chi_{\ran(\phi)}(x)b(x)f(\phi^{-1}x)=\]
\[T(f\circ \phi^{-1})(x).\]

	Thus we may assume that $\pi(x,y)=\id$ for all $(x,y)\in \R.$ 
 Find $A\subseteq X$ so that up to sets of measure zero,
\[X=\bigsqcup_{j=0}^{n-1}\alpha^{j}(A).\]

	Let $S=(e_{j}\otimes \chi_{A})_{j=1}^{n},$ where $v\otimes f(x)=f(x)v$ for $f\colon X\to \CC$ measurable and $v\in \CC^{k}.$ Set
\[\rho_{i}(f)=\sum_{j=1}^{k}\|f_{j}\|,\]
for $f\in l^{\infty}(k,l^{p}(d_{i})).$  Fix $\Phi\subseteq [[\R]]$ containing $\{\id,\alpha,\alpha^{2},\cdots,\alpha^{n-1}\}$ and a set $\mathcal{P}$ of projections in $L^{\infty}(A,\mu)$ so that there is a sequence $\mathcal{P}_{n}$ of partitions of $A$ in $\mathcal{P}$ so that $\mathcal{P}_{n}\to \id.$ Without loss of generality, we may assume that for each $n,$ $\sigma_{i}$ is eventually a $*$-homomorphism on $W^{*}(\mathcal{P}_{n},\alpha)$ with $\tr(\sigma_{i}(\id))\to 1.$ Let $\mathcal{P}_{n}=\{B_{1,n},\cdots,B_{m_{n},n}\}.$  

	Fix $\varepsilon>0,$ and $N\in \NN.$  Suppose $F\subseteq \Phi$ is finite, and contains $\{\id,\alpha,\alpha^{2},\cdots,\alpha^{n-1},\id_{A}\},$ $m\in \NN,\delta>0.$ It is easy to see that $\alpha_{S}(\Hom_{\R,l^{p}}(S,F,m,\delta,\sigma_{i}))$ is almost contained in $l^{p}(\sigma_{i}(\id_{A})\{1,\cdots,d_{i}\})^{\oplus k}.$ Thus
\[\dim_{\Sigma,l^{p}}(F,m,\delta,\varepsilon,\rho_{i})\leq \lim_{i\to \infty}\frac{k\tr(\sigma_{i}(\id_{A})}{d_{i}}=\frac{k}{n}.\]

	Define
\[T_{\xi,N}\colon L^{p}(X,\mu,\CC^{k})\to l^{p}(d_{i})\]
by
\[T_{\xi,N}(f)=\sum_{j=0}^{n-1}\sum_{k=1}^{m_{N}}\left(\frac{1}{\mu(B_{k,N})}\int_{B_{k,N}}f\circ \alpha^{j}\,d\mu\right)\sigma_{i}(\alpha)^{j}\sigma_{i}(\id_{B_{k,N}})\xi.\]

	Simple estimates prove that 
\[\|T_{\xi,N}(f)\|_{p}^{p}\leq \sum_{j=0}^{n-1}\sum_{k=1}^{m}\left(\int_{\alpha^{j}(B_{k,N})}|f|^{p}\,d\mu\right)\frac{\|\sigma_{i}(\id_{B_{k,N}})\xi\|_{p}^{p}}{\mu(B_{k})^{p}}.\]

	As
\[\frac{1}{\vol(\Ball(l^{p}(d_{i},\nu_{i}))}\int_{\Ball(l^{p}(d_{i},\nu_{i})}\|\sigma_{i}(\id_{B_{k,N}})\xi\|_{p}\,d\mu=\tr(\sigma_{i}(\id_{B_{k,N}}))\to \mu(B_{k,N}),\]

	there are $C_{i}\subseteq \Ball(l^{p}(d_{i},\nu_{i})$ with $\liminf_{i}\frac{\vol(C_{i})}{\vol(\Ball(l^{p}(d_{i},\nu_{i}))}\geq 1/3,$ so that $\|T_{\xi,N}\|\leq 2$ if $\xi\in C_{i}.$ 

	For all large $i,$
\[T_{\xi,N}(\alpha f)=\sigma_{i}(\alpha)T_{\xi}(f).\]
\[T_{\xi,N}(\id_{B_{k,N}}f)=\sigma_{i}(\id_{B_{k,N}})T_{\xi,N}(f)\]
thus if $N$ is large enough $T_{\xi,N}\in \Hom_{\R,l^{p}}(S,F,m,\delta,\sigma_{i})$ if $\|\xi\|_{p}\leq 1.$

	As 
\[T_{\xi,N}(\chi_{A})=\sigma_{\id}(A)\xi,\] we have
\[\alpha_{S}(\Hom_{\R,l^{p}}(S,F,m,\delta,\sigma_{i}))\supseteq \{\sigma_{\id}(A)\xi:\xi \in C_{i}\}.\]
so
\[\dim_{\Sigma,l^{p}}(L^{p}(X,\mu,\CC^{k}),\R)\geq \frac{k}{n}.\]

\end{proof}

\begin{cor}\label{C:finitedim} Let $(X,\mu,\R)$ be a discrete measure-preserving equivalence relation. Suppose that $\mathcal{O}_{x}$ is infinite for almost every $x.$ Let $V$ be a measureable field of finite-dimensional vector spaces with an action of $\R.$ Then for $1\leq p<\infty,$ we have
\[\dim_{\Sigma,l^{p}}\left(\int_{X}^{\oplus_{p}}V_{x}\,d\mu(x),\R\right)=0.\]
\end{cor}

\begin{proof} This is simple from the preceding propositionsince for every $n\in \NN,$ there is a subequivalence relation $\R_{n}\subseteq \R,$ where $\R_{n}$ has orbits of size $n$ for almost every $x\in X.$

\end{proof}

\section{$l^{p}$-Cohomology of Equivalence Relations}\label{S:cohom}

	Let $G$ be a\ locally finite graph. We let $\mathcal{E}(G)$ be the set of oriented edges of $G,$ and $E(G)$ the set of unoriented edges of $G,$ also we let  $V(G)$ be the set of vertices of $G.$  If $x,y\in V(G),$ we let $(x,y)$ be the oriented edge from $x$ to $y,$ and $[x,y]$ be the unoriented edge between $x$ and $y.$ We shall abuse notation and use $\CC^{E(G)}$ for all functions $f\colon \mathcal{E}(G)\to \CC$ such that $f(x,y)=-f(y,x)$ for all $(x,y)\in \mathcal{E}(G).$ We let $l^{p}(E(G))$ be the the functions in $\CC^{E(G))}$ so that 
\[\|f\|_{p}^{p}=\sum_{[x,y]\in E(G)}|f(x,y)|^{p}<\infty\]
(note $|f(x,y)|$ does not depend on the orientation of $[x,y].$ ) Similar remarks apply for $c_{c}(E(G))$ and other function spaces. 

	If $e=(x,y)$ is an oriented edge in $G,$ define $\mathcal{E}_{x,y}(u,v)=0$ if one of $u,v$ is not $x$ or $y,$ $1$ if $(u,v)=(x,y)$ and $-1$  if $(u,v)=(y,x).$ If $\gamma\colon \{0,\cdots,k\}\to V(G)$ is a path (that is $\gamma(j-1),\gamma(j)$ are adjacent) we think of $\gamma$ as a an element of $l^{p}(E(G))$ by having $\gamma$ correspond to
\[\sum_{j=1}^{k}\mathcal{E}_{(\gamma(j-1),\gamma(j))}.\]

	For $f\colon \mathcal{E}(G)\to \CC$ and $\gamma$ a path as above, we define
\[\int_{\gamma}f=\sum_{j=1}^{k}f(\gamma(j-1),\gamma(j)).\]

For a general graph $G,$ define $\delta\colon \CC^{V(G)}\to \CC^{E(G))},$ $\partial\colon \CC^{E(G)}\to \CC^{V(G)}$ by 
\[\delta f(v,w)=f(w)-f(v)\]
\[(\partial f)(v)=\sum_{w\mbox{ adjacent  to } v}f(w,v).\]

	Then $\delta$ and $\partial$ are dual in the following sense: if $f\in c_{c}(E(G)),$ and $g\in \CC^{V(G)}$ then 
\[\ip{\partial f,g}=-\ip{f,\delta g}\]
	where 
\[\ip{h,k}=\sum_{[x,y]\in E(G)}h(x,y)k(x,y)\]
for $h\in c_{c}(E(G)),k\in \CC^{E(G)},$ (again this is independent of orientation). Similarly if $f\in \CC^{E(G)},g\in c_{c}(V(G)),$ then
\[\ip{g,\partial f}=-\ip{\delta g,f}.\]

	Let
\[B_{1}(G)=\Span\{\gamma\in c_{c}(E(G)):\gamma\mbox{ is a loop}\}.\]
\[Z_{1}(G)=\{f\in c_{c}(E(G)):\partial f=0\}\]
\[Z^{1}(G)=\left\{f\in \CC^{E(G)}:\int_{\gamma}f=0\mbox{ for all loops $\gamma$}\right\}\]
If $f\in Z^{1}(G)$ and $v,w$ are vertices in $G,$ and $\gamma\colon \{0,\cdots,k\}\to V(G)$ is a path from $v$ to $w,$ then
\[\int_{\gamma}f\]
depends only on $v$ and $w$ since $f$ integrates to zero along all loops. We will use
\[\int_{v\to_{G}w}f,\]
for this number. Note that $Z^{1}(G)=\{\delta_{G}h:h\in \CC^{V(G)}\}.$ In fact, if $f\in Z^{1}(G),$ and $G_{j}$ are the connected components, then for   fixed $x_{j}\in V(G_{j})$ 
\[h(v)=\int_{x_{j}\to_{G}v}f,\]
for $v\in V(G_{j})$ has $\delta_{G}h=f.$ 

	Define the space of $l^{p}$-cocycles by
\[Z^{1}_{(p)}(G)=Z^{1}(G)\cap l^{p}(E(G)).\]
	We let 
\[B^{(p)}_{1}(G)=\overline{B_{1}(G)}^{\|\cdot\|_{p}},\]
be the space of $l^{p}$-boundaries.

	If $G'\subseteq G$ is a subgraph we identify $\CC^{E(G')}\subseteq \CC^{E(G)}$ by extending by zero. This allows us to make sense of all the function spaces above for $G'$  as subsets of $\CC^{E(G)}.$

\begin{definition} \emph{ Let $(\R,X,\mu)$ be a discrete measure-preserving equivalence relation. A} measurable field of graphs fibered over $\R$ \emph{is a field $\{\Phi_{x}\}_{x\in X}$ of graphs  having vertex set $\mathcal{O}_{x},$ such that $\Phi_{x}=\Phi_{y}$ for almost every $(x,y)\in \R,$ and  $\bigcup_{x\in X}\mathcal{E}(\Phi_{x})$ is a measurable subset of $\R$ which intersects the diagonal in a set of measure zero.}\end{definition}

	We set $\mathcal{E}(\Phi)=\bigcup_{x\in X}\mathcal{E}(\Phi_{x}).$

	If $\Phi$ is a measurable field of graphs fibered over $\R,$ we define the cost of $\Phi$ by (see \cite{L}, \cite{Gab1} for important properties of cost)
\[c(\Phi)=\frac{1}{2}\int_{X}\deg(x)\,d\mu(x)\]
where $\deg(x)$ is the degree of the vertex $x.$ This is also
\[\frac{1}{2}\overline{\mu}(\Phi).\]

	For any graphing $\Phi=(\phi_{j})_{j\in J},$ and for each $x\in X,$ we define a graph whose vertices are $\mathcal{O}_{x}$ and whose oriented edges are $\{(u,v):u\thicksim x,v=\phi^{\pm 1}(u),\mbox{ for some $\phi\in \Phi$}\}.$ If $\Phi_{x}$ denotes the corresponding graph note that 
\[c(\Phi)=\sum_{j\in J}\mu(\dom(\phi_{j})).\]

	This is simply the cost of the graphing $\Phi$ as previously defined. It is also straightforward to check that any measurable field of graphs over $X$ comes from a graphing of a subequivalence relation.

 If $x\to \Phi_{x}$ is a measurable field of graphs over $X,$ let $L^{p}(G)/B_{1}^{(p)}(G)$ be the $L^{p}$-direct integral of the space $l^{p}(E(\Phi_{x}))/B_{1}^{(p)}(\Phi_{x}).$ Note that $\R$ has a representation $\pi$ on $l^{p}(E(\Phi_{x}))/B_{1}^{(p)}(\Phi_{x}),$ given by $\pi(x,y)=\id$ for all $(x,y)\in \R.$

  We will show that if $\R$ has finite cost and satisfies a ``finite presentation" assumption, then $\dim_{\Sigma}(L^{p}(E(\Phi))/B_{1}^{(p)}(E(\Phi)),\R)$ does not depend on the choice of finite cost graph $\Phi.$

\begin{definition}\emph{ Let $(\R,X,\mu)$ be a discrete measure-preserving equivalence relation. Let $\Phi=(\phi_{j})_{j\in J}$ be a graphing of $\R.$ We say that $\Phi$ is} finitely presented \emph{ if there are measurable fields of loops $(L^{(j)})_{j=1}^{\infty}$ such that for almost every $x\in X,$}\end{definition}
$\Span\{L^{(j)}_{y}:y\thicksim x,j\in \NN\}=B_{1}(\Phi_{x}),$
and
\[\sum_{j=1}^{\infty}\mu(\supp L^{(j)})<\infty.\]

	We say that $\R$ is \emph{ finitely presented} if it has a finitely presented graphing.

	For example, if $\R$ is a induced by a free action of a finitely presented group, then $\R$ is finitely presented.
	
	We will proceed to show that if $\R$ is finitely presented, then in fact \mbox{every} graphing is finitely presented. It may be useful to consider the group analouge first.

	Suppose $\Gamma=\ip{s_{1},\cdots,s_{n}|r_{1},r_{2},\cdots,r_{m}}$ is a finitely presented group. And suppose that $t_{1},\cdots,t_{k}$ also generate $\Gamma.$ Choose words $w_{i}$ in $t_{1},\cdots,t_{k}$ so that 
\[w_{i}(t_{1},\cdots,t_{k})=s_{i}\]
and choose words $v_{i}$ in $s_{1},\cdots,s_{n}$ so that 
\[t_{i}=v_{i}(s_{1},\cdots,s_{n}).\]

	Set 
\[\sigma_{i}=r_{i}(w_{1},\cdots,w_{n}),\]
\[\eta_{i}=v_{i}(w_{1},\cdots,w_{n}),\]
then one can show that
\[\Gamma=\ip{t_{1},t_{2},\cdots,t_{k}|\sigma_{1},\cdots,\sigma_{m},\eta_{1}t_{1}^{-1},\eta_{2}t_{2}^{-1},\cdots,t_{k}a_{k}^{-1}}.\]

	 Graphically, choosing words $w_{i},v_{i}$ as above corresponds to finding a path in $\Cay(\Gamma,\{t_{1},\cdots,t_{k}\})$ from $e$ to $s_{i}$ and vice versa. So we will simply express the above proof in the language of graphs and this will allow us to generalize to the case of equivalence relations.

\begin{lemma} Let $G,G'$ be two connected locally finite graphs with the same vertex set. Choose paths $\{\sigma_{y,z}\}_{(y,z)\in \mathcal{E}(G)}$ in $G'$ from $y$ to $z$ such that $\sigma_{yz}=-\sigma_{zy}.$ Similarly, choose paths $\{\gamma_{v,w}\}_{(v,w)\in \mathcal{E}(G')}$ in $G$ from $v$ to $w$ such that $\gamma_{vw}=-\gamma_{wv}.$ Suppose that $\{L_{j}:j\in J\}$ is a family of loops in $G$ so that 
\[B_{1}(G)=\Span\{L_{j}:j\in J\}.\]

	Define $T\colon c_{c}(E(G))\to c_{c}(E(G')),$ by
\[Tf=\sum_{[y,z]\in E(G)}f(y,z)\sigma_{yz},\]
	Then 
\[B_{1}(G')=\Span\{T(L_{j}):j\in J\}+\Span\{T(\gamma_{v,w})-\mathcal{E}_{(v,w)}:(v,w)\in \mathcal{E}(G')\}.\]

\end{lemma}

\begin{proof} Note that 
\[c_{c}(E(G))=\bigcup_{F\subseteq E(G) \mbox{ finite }}c_{c}(F),\]
give $c_{c}(E(G))$ the direct limit topology with respect to this filtration. That is, if $f_{n}\in c_{c}(E(G))$ then $f_{n}\to f\in c_{c}(E(G))$ if and only if there is a finite subset $F\subseteq E(G)$ so that $\supp\{f_{n}\}\subseteq F$ and $f_{n}\to f$ pointwise. It is easy to see that every subspace of $c_{c}(E(G))$ is closed in this topology, and that $c_{c}(E(G))^{*}=\CC^{E(G)}$ with respect to the pairing
\[\ip{f,g}=\sum_{[y,z]\in E(G)}f(y,z)g(y,z)\]
(the above sum being independent of the orientation of edges).

 Let $g\in \CC^{E(G')}$ be such 
\[\int_{T(L_{j})}g=0,g(v,w)=\int_{(T(\gamma_{v,w}))}g.\]

	Note that 
\[T^{t}\colon \CC^{E(G')}\to \CC^{E(G)},\]
is given by
\[T^{t}f(y,z)=\int_{\sigma_{yz}}f,\]

	Thus
\[\int_{L_{j}}T^{t}g=\int_{T(L_{j})}g=0,\]
for all $j.$ This implies that there is a $h\colon V(G)\to \CC$ such that $\delta_{G}h=T^{t}g.$ Note that for all $(v,w)\in \mathcal{E}(G'),$
\[h(w)-h(v)=\int_{\gamma_{vw}}T^{t}g=\int_{T(\gamma_{v,w})}g=g(v,w).\]

	Therefore,
\[\delta_{G'}h=g.\]

	This implies that $g\in Z^{1}(G').$ The Hahn-Banach Theorem now completes the proof.

\end{proof}
\begin{lemma} Let $(\R,X,\mu)$ be a discrete measure-preserving equivalence relation, if $\R$ is finitely presented then every finite cost graphing of $\R$ is finitely presented.
\end{lemma}

\begin{proof} Let $\Phi$ be finitely presented and let $L^{(j)}$ be as in the definition of finitely presented. Let $(\mathcal{E}_{k})_{k\in K}$ be a countable family of  partially defined measurable functions from $X$ to $X$  with the following properties:

\begin{list}{\arabic{pcounter}:~}{\usecounter{pcounter}}
\item $\mathcal{E}(\Phi)=\bigcup_{k\in K}\{(x,\mathcal{E}_{k}(x)):x\in \dom(\mathcal{E}_{k})\}\cup \{(\mathcal{E}_{k}(x),x):x\in \dom(\mathcal{E}_{k})\}$
\item for all $j,k$ $\{(x,\mathcal{E}_{j}(x)):x\in \dom(\mathcal{E}_{j})\}\cap \{(\mathcal{E}_{k}(x),x):x\in\dom(\mathcal{E}_{k})\}=\varnothing$
\item for all $j\ne k$,$ \{(x,\mathcal{E}_{j}(x)):x\in \dom(\mathcal{E}_{j})\}\cap \{(x,\mathcal{E}_{k}(x)):x\in\dom(\mathcal{E}_{k})\}=\varnothing$
\end{list}

	For each $k\in K,$ let $\sigma^{(k)}_{x}$ be a measurable family of paths  in $\Psi$ so that for almost every $x,$ $\sigma^{(k)}_{x}$ is a path form $x$ to $\mathcal{E}_{k}(x).$ Define
\[T_{x}\colon c_{c}(E(\Phi_{x}))\to c_{c}(E(\Psi_{x}))\]
by
\[T_{x}f=\sum_{y\thicksim x}\sum_{k\in K:y\in \dom(\mathcal{E}_{k})}f(y,\mathcal{E}_{k}(y))\gamma^{(k)}_{y}.\]

	Then $T_{x}=T_{y}$ if $y\thicksim x.$ Let $(\mathcal{D}_{\alpha})_{\alpha\in A}$ be a countable family of partially defined measurable functions from $X$ to $X$ in $\Psi$ following properties:
\begin{list}{\arabic{pcounter}:~}{\usecounter{pcounter}}
\item $\mathcal{E}(\Psi)=\bigcup_{k\in A}\{(x,\mathcal{D}_{k}(x)):x\in \dom(\mathcal{D}_{k})\}\cup \{(\mathcal{D}_{k}(x),x):x\in \dom(\mathcal{D}_{k})\}$
\item for all $j,k$ $\{(x,\mathcal{D}_{j}(x)):x\in \dom(\mathcal{D}_{j})\}\cap \{(\mathcal{D}_{k}(x),x):x\in\dom(\mathcal{E}_{k})\}=\varnothing$
\item for all $j\ne k$,$ \{(x,\mathcal{D}_{j}(x)):x\in \dom(\mathcal{D}_{j})\}\cap \{(x,\mathcal{D}_{k}(x)):x\in\dom(\mathcal{D}_{k})\}=\varnothing$
\end{list}

	Let $\gamma^{(\alpha)}_{x}$ be a measurable family of paths in $\Phi$ so that for almost $x,$ $\gamma^{(\alpha)}_{x}$ is a path from $x$ to $\mathcal{D}_{\alpha}(x).$ From the preceding lemma, it then follows that $(T_{x}L^{(j)}_{x})_{j=1}^{\infty},(T_{x}(\gamma^{(\alpha)}_{x})-\mathcal{E}_{(x,\mathcal{D}^{(\alpha)}_{x})}$ is a measurable family of loops in $\Psi_{x}$ whose $\R$-translates span $B_{1}(\Psi_{x}).$ Further,
\[\sum_{j}\mu(\supp T( L^{(j)}))\leq \sum_{j=1}^{\infty}\mu(\supp L^{(j)})<\infty,\]
\[\sum_{\alpha}\mu(\supp (T(\gamma^{(\alpha)})-\mathcal{E}_{(\cdot,\mathcal{D}_{\alpha}(\cdot))})\leq c(\Psi)<\infty.\]

\end{proof}

	We now proceed to prove  that $\dim_{\Sigma,l^{p}}(L^{p}(E(\Phi))/B_{1}^{(p)}(\Phi),\R)$ does not depend upon the choice of finite cost graphing when $\R$ is finitely presented. Our methods are similar to Gaboriau's in \cite{Gab2}. We must be more careful, however, since we do not have monotonicity of our dimension. We will need the following ``Continuity Lemma."

\begin{lemma}\label{L:finprescontinuity} Fix $1\leq p,q<\infty.$ Let $(\R,X,\mu)$ be as before, with $\R$ finitely presented. If $\Phi$ is a finite cost graphing of $\R,$ and $\Phi^{(n)}$ is an increasing sequence of subgraphs of $\R$ so that 
\[\Phi_{x}=\bigcup_{n=1}^{\infty}\Phi^{(n)}_{x}\]
for almost every $x,$ then
\[\dim_{\Sigma,l^{q}}(L^{p}(E(\Phi^{(n)}))/B_{1}^{(p)}(\Phi^{(n)}),\R)\to \dim_{\Sigma}(L^{p}(E(\Phi))/B_{1}^{(p)}(\Phi),\R),\]
\[\dim_{\Sigma,l^{q}}(B_{1}^{(p)}(E(\Phi))/B_{1}^{(p)}(\Phi^{(n)}),\R)\to 0,\]
\[\underline{\dim}_{\Sigma,l^{q}}(L^{p}(E(\Phi^{(n)}))/B_{1}^{(p)}(\Phi^{(n)}),\R)\to \underline{\dim}_{\Sigma}(L^{p}(E(\Phi))/B_{1}^{(p)}(\Phi),\R),\]
\[\underline{\dim}_{\Sigma,l^{q}}(B_{1}^{(p)}(E(\Phi))/B_{1}^{(p)}(\Phi^{(n)}),\R)\to 0.\]

\end{lemma}

\begin{proof} Let $E\colon L^{p}(E(\Phi^{(n)})) \to L^{p}(E(\Phi))$ be defined by extension by zero. It is easy to see that $E$ descends to a well-defined map, still denoted $E$
\[L^{p}(E(\Phi^{(n)}))/B_{1}^{(p)}(\Phi^{(n)}) \to L^{p}(E(\Phi))/B_{1}^{(p)}(\Phi).\]
By subadditivity under exact sequences,
\begin{align*}
\dim_{\Sigma,l^{q}}(L^{p}(E(\Phi))/B_{1}^{(p)}(\Phi),\R)&\leq \dim_{\Sigma,l^{q}}L^{p}(E(\Phi^{(n)}))/B_{1}^{(p)}(\Phi^{(n)}),\R)\\
&+\dim_{\Sigma,l^{q}}([L^{p}(E(\Phi^{(n)}))/B_{1}^{(p)}(\Phi^{(n)})]/\overline{\im E},\R),
\end{align*}
and it is easy to see that there is a $\R$-equivariant map
\[L^{p}(E(\Phi\setminus \Phi^{(n)}))\to [L^{p}(E(\Phi^{(n)}))/B_{1}^{(p)}(\Phi^{(n)})]/\overline{\im E}\]
with dense image. Thus
\begin{align*}
\dim_{\Sigma,l^{q}}(L^{p}(E(\Phi))/B_{1}^{(p)}(\Phi),\R)&\leq \dim_{\Sigma,l^{q}}(L^{p}(E(\Phi^{(n)}))/B_{1}^{(p)}(\Phi^{(n)}),\R)\\
&+c(\Phi\setminus \Phi^{(n)}).
\end{align*}
 And this proves one side of the necessary inequality.

	For the opposite inequality, consider the restriction map
\[R\colon L^{p}(E(\Phi))\to L^{p}(E(\Phi^{(n)})),\]
 then $R$ descends to an  surjective $\R$-equivariant map
\[L^{p}(E(\Phi))/B_{1}^{(p)}(\Phi^{(n)})\to L^{p}(E(\Phi))/B_{1}^{(p)}(\Phi^{(n)}).\]

	Thus

\[\dim_{\Sigma,l^{q}}(L^{p}(E(\Phi))/B_{1}^{(p)}(\Phi^{(n)}),\R)\leq \dim_{\Sigma}(L^{p}(E(\Phi))/B_{1}^{(p)}(\Phi^{(n)}),\R).\]

Considering the exact sequence

\[\begin{CD} 0@>>> \frac{B^{(p)}_{1}(\Phi)}{B^{(p)}_{1}(\Phi^{(n)})} @>>> \frac{L^{p}(\Phi)}{B^{(p)}_{1}(\Phi^{(n)})} @>>> \frac{L^{p}(\Phi)}{B^{(p)}_{1}(\Phi)} @>>> 0,\end{CD}\]
we find that 
\begin{align*}
\dim_{\Sigma,l^{q}}(L^{p}(E(\Phi))/B_{1}^{(p)}(\Phi^{(n)}),\R)&\leq \dim_{\Sigma,l^{q}}\left(\frac{B^{(p)}_{1}(\Phi)}{B^{(p)}_{1}(\Phi^{(n)})},\R\right)\\
&+\dim_{\Sigma,l^{q}}\left(\frac{L^{p}(E(\Phi))}{B^{(p)}_{1}(\Phi)},\R\right) .
\end{align*}

	So it suffices to prove the second limiting statement. For this, since $\R$ is finitely presented we can find measurable fields of loops $(L^{(j}))_{j=1}^{\infty}$ which generate $B_{1}^{(p)}(\Phi)$ and so that 
\[\sum_{j=1}^{\infty}\mu(\supp L^{(j)})<\infty.\]

	Since
\[\dim_{\Sigma,l^{q}}(L^{p}(E(\Phi))/B_{1}^{(p)}(\Phi^{(n)}),\R)\leq \sum_{j=1}^{\infty}\mu(\{x:L^{(j)}_{x}\mbox{ is not supported in } \Phi^{(n)}_{x}\})\]
and
\[\mu(\{x:L^{(j)}_{x}\mbox{ is not supported in } \Phi^{(n)}_{x}\})\to 0,\]
\[\mu(\{x:L^{(j)}_{x}\mbox{ is not supported in } \Phi^{(n)}_{x}\})\leq \mu(\supp L^{(j)}),\]
we find that 
\[\dim_{\Sigma,l^{q}}(L^{p}(\Phi)/B_{1}^{(p)}(\Phi^{(n)}),\R)\to 0,\]
as desired.

\end{proof}

\begin{theorem} Fix $1\leq p,q<\infty.$ Let $(\R,X,\mu)$ be as before with $\R$ finitely presented and of  finite cost. Let $\Phi,\Psi$ be two finite cost graphings of $\R.$ Then
\[\dim_{\Sigma,l^{q}}(L^{p}(E(\Phi))/B_{1}^{(p)}(\Phi),\R)=\dim_{\Sigma,l^{q}}(L^{p}(E(\Psi))/B_{1}^{(p)}(\Psi),\R),\]
\[\underline{\dim}_{\Sigma,l^{q}}(L^{p}(E(\Phi))/B_{1}^{(p)}(\Phi),\R)=\underline{\dim}_{\Sigma,l^{q}}(L^{p}(E(\Psi))/B_{1}^{(p)}(\Psi),\R),\]
\end{theorem}

\begin{proof} Let $\Phi=(\phi_{j})_{j=1}^{\infty}.$ Let $\Phi^{(n)}_{x},\Psi^{(nm)}_{x}$ be the subgraphs defined by
\[\mathcal{E}(\Phi^{(n)}_{x})=\{(y,\phi_{j}^{\pm 1}(y)):1\leq j\leq n,y\in \dom(\phi_{j}^{\pm 1}),y\thicksim x\}\]
\[\mathcal{E}(\Psi^{(n,m)}_{x})=\{(y,z)\in \mathcal{E}(\Psi_{x}):d_{\Phi^{(n)}_{x}}(y,z)\leq m\}.\]

	Note that if $\gamma_{yz},\gamma_{y,z}'$ are two paths from $y$ to $z$ in $\Phi^{(n)},$ then their difference is a loop in $\Phi^{(n)}.$ Thus for $(y,z)\in \mathcal{E}(\Psi^{(n,m)}_{x})$ we have a well-defined element $\sigma_{yz}$ of $l^{p}(E(\Phi^{(n)})_{x})/B_{1}^{(p)}(\Phi^{(n)}_{x})$ given as the equivalence class of any path from $y$ to $z$ in $\Phi^{(n)}.$ 
	
	Then  for each $n,m$ we have a well-defined bounded linear map with $T_{x}$ (whose norm is bounded uniformly in $x$)
\[T_{x}\colon l^{p}(E(\Psi^{(n,m)}_{x}))/B_{1}^{(p)}(\Psi^{(n,m)})\to l^{p}(E(\Phi^{(n)}_{x}))/B_{1}^{(p)}(\Phi^{(n)})\ \]
by
\[T_{x}f=\sum_{[y,z]\in E(\Phi^{(n)})}f(y,z)\sigma_{yz},\]
 Let
\[T=\int^{\oplus}_{X}T_{x}\,d\mu(x),\]
then $T$ is an $\R$- equivariant map
\[L^{p}(\Psi^{(n,m)})/B^{(p)}_{1}(\Psi^{(n,m)})\to L^{p}(\Phi^{(n)})/B^{(p)}_{1}(\Phi^{(n)}).\]

	Thus
\begin{align*}
\dim_{\Sigma,l^{q}}(L^{p}(E(\Phi^{(n)}))/B^{(p)}_{1}(\Phi^{(n)}),\R)&\leq \dim_{\Sigma,l^{q}}(\overline{\im T},\R)\\
&+\dim(L^{p}(E(\Phi^{(n)}))/B^{(p)}_{1}(\Phi^{(n)})/\overline{\im T},\R)\\
&\leq \dim_{\Sigma,l^{q}}(L^{p}((E\Psi^{(n,m)}))/B^{(p)}_{1}(\Psi^{(n,m)}),\R)\\
&+\dim_{\Sigma,l^{q}}(L^{p}(E(\Phi^{(n)}))/B^{(p)}_{1}(\Phi^{(n)})/\overline{\im T},\R).
\end{align*}

	Now suppose that $x\thicksim y\thicksim z$ in $X,$ and $y,z$ are in the same connected component in $\Psi^{(n,m)}_{x}.$ Then we can find $x_{1},\cdots,x_{n}$ with $y=x_{1},x_{n}=z$ which are adjacent and 
\[\sum_{i=1}^{n-1}\sigma^{x_{i}x_{i+1}}_{x},\]
is a path from $y$ to $z$ in $B_{x}.$ Further, if $\sigma_{yz}$ is any other such path, then again there difference is a loop, so $\sigma_{yz}$ represents a well-defined element in $\im(T).$ Let
\[Y^{(n)}_{x}=\overline{\Span}^{\|\cdot\|_{p}}\{\sigma_{yz}:y,z\mbox{ are connected in } \Psi^{(n,m)}_{x}\}.\]

	Then
\[\dim_{\Sigma}(L^{p}(E(\Phi^{(n)}))/B^{(p)}_{1}(\Phi^{(n)})/\overline{\im T},\R)\leq \dim_{\Sigma}(L^{p}(E(\Phi^{(n)}))/B^{(p)}_{1}(\Phi^{(n)})/Y_{n},\R).\]

	Now let $V^{(n)}_{x}\subset l^{p}(G^{B}_{x})$ be defined by
\[V^{(n)}_{x}=\overline{\Span}^{\|\cdot\|_{p}}\{\gamma^{yz}:\gamma^{yz} \mbox{ is a path from $y$ to $z$ in $\Phi^{(n)}_{x},$ $y,z$ connected in $\Psi^{(n,m)}_{x}$}\}.\]

	Then we have a surjective equivariant map
\[L^{p}(E(\Phi^{(n)}))/V^{(n)}\to (L^{p}(E(\Phi^{(n)}))/B^{(p)}_{1}(\Phi^{(n)})/Y_{n},\]
so
\[\dim_{\Sigma}(L^{p}(E(\Phi^{(n)}))/B^{(p)}_{1}(\Phi^{(n)})/Y_{n},\R)\leq \dim_{\Sigma}(L^{p}(E(\Phi^{(n)}))/V^{(n)},\R).\]

	Let $(E_{j})_{j=1}^{\infty}$ be disjoint edges generating $L^{p}(E(\Phi^{(n)}))$ such that
\[\sum_{j=1}^{\infty}\mu(\supp(E_{j}))=c(\Phi^{(n)}).\]

	Writing $E^{(j)}_{x}=(f(x),g(x)).$ Then
\begin{align*}
\dim_{\Sigma,l^{q}}(L^{p}(E(\Phi^{(n)}))/V^{(n)},\R)&\leq\sum_{j=1}^{\infty}\mu(\{x\in \supp(E_{j}):(f(x),g(x))\notin \mathcal{C}(\Psi^{(m,n)}_{x})\})\\
&=c(\Phi^{(n)}\setminus \mathcal{C}(\Psi^{(n,m)}))
\end{align*}

where
\[\mathcal{C}(\Psi^{(n,m)})=\{(y,z)\in \R: y\mbox{ is connected to $z$ in } \Phi^{(n)}_{x}\}.\]

	Putting this altogether we have

\begin{align*}
\dim_{\Sigma,l^{q}}(L^{p}(E(\Phi^{(n)}))/B^{(p)}_{1}(\Phi^{(n)}),\R)&\leq\dim_{\Sigma,l^{q}}(L^{p}(\Psi^{(n,m)})/B^{(p)}_{1}(\Psi^{(n,m)}),\R)\\
&+c(\Phi^{(n)}\setminus \mathcal{C}(\Psi^{(n,m)})),
\end{align*}

choose an increasing sequence of integers $m_{n}$ so that 
\[c(\Psi_{x}\cap \mathcal{C}(\Phi^{(n)})) \setminus \Psi^{(n,m_{n})})\to 0\]
\[c([\Phi^{(n)}\setminus \mathcal{C}(\Psi_{x}\cap \mathcal{C}(\Phi^{(n)}_{x}))]\setminus[\Phi^{(n)}\setminus \mathcal{C}(\Psi^{(n,m_{n})})])\to 0.\]

	Then $\Psi^{(n,m_{n})}$ increases to $\Psi,$ and it is easy to see that 
\[c(\Phi^{(n)}\setminus \mathcal{C}(\Psi^{(n,m_{n})}))\to 0.\]

	Thus letting $n\to \infty$ and applying the preceding lemma we find that 
\[\dim_{\Sigma,l^{q}}(L^{p}(E(\Phi))/B^{(p)}_{1}(\Phi),\R)\leq \dim_{\Sigma,l^{q}}(L^{p}(E(\Psi))/B^{(p)}_{1}(\Psi),\R)\]
the proposition now follows by symmetry.

\end{proof}

\begin{definition} \emph{Let $(X,\mu,\R)$ be a discrete measure-preserving equivalence relation, with $\R$ finitely presented and of finite cost, and let $\Sigma$ be a sofic approximation of $\R.$  By the above Theorem, the number $c^{(p)}_{1,\Sigma}(\R)=\dim_{\Sigma,l^{p}}(L^{p}(E(\Phi))/B^{(p)}_{1}(\Phi),\R)$ is independent on the choice of a finite cost graphing $\Phi.$ Similar remarks apply to ~$\underline{c}^{(p)}_{1,\Sigma}(\R)=\underline{\dim}_{\Sigma,l^{p}}(L^{p}(E(\Phi))/B^{(p)}_{1}(\Phi),\R).$~}\end{definition}

	It is easy to see that $c^{(p)}_{1}(\R)\leq c(\R).$ By Theorem $\ref{T:HSlowerbound},$  if $\R$ has infinite orbits we then 
\[\beta^{(2)}_{1}(\R)+1=\underline{c}^{(2)}_{1,\Sigma}(\R)= c^{(2)}_{1,\Sigma}(\R)\leq c(\R),\]
and a well-known conjecture would imply that this inequality is an equality.

\begin{theorem} Let $(X,\mu,\R)$ be a ergodic, finitely presented, discrete, measure-preserving equivalence relation, and let $\Sigma$ be a sofic approximation of $\R.$ Let $A\subseteq \R,$ and define $\sigma_{i,A}\colon L(\R_{A})\to M_{d_{i}}(\CC)$ by $\sigma_{i,A}(x)=\sigma_{i}(\id_{A})\sigma_{i}(x)\sigma_{i}(\id_{A}).$ Then
\[\mu(A)(c^{(p)}_{1,\Sigma_{A}}(\R_{A})-1)\geq c^{(p)}_{1,\Sigma}(\R)-1.\]
\end{theorem}
\begin{proof} Let $\Psi$ be a graphing of $\R_{A}.$ Let $n\in \NN\cup\{0\}$ be such that $n\mu(A)\leq 1<(n+1)\mu(A).$ Let $A=A_{1},A_{2},\cdots,A_{n}$ be essentially disjoint measurable sets such that there exists $\phi_{i}\in[[\R]],$ with $\dom(\phi_{i})=A,\ran(\phi_{i})=A_{i},$  and let $A'\subseteq A$ be such that there is $\phi_{n+1}\in [[\R]]$ with $\dom(\phi_{n+1})=A',$ and 
\[\ran(\phi_{n+1})=X\setminus \bigcup_{j=1}^{n}A_{j}.\]

	Let $\Phi=\Psi\cup\{\phi_{j}\}_{j=1}^{n+1}.$ We use $L^{p}(E(\Psi\big|_{A})),B_{1}^{(p)}(\Psi\big|_{A}),$ for 
\[\int_{A}^{\oplus_{p}}l^{p}(E(\Psi)_{x})\,d\mu(x),\]
\[\int_{A}^{\oplus_{p}}B_{1}^{(p)}(\Psi_{x})\,d\mu(x),\]
and $L^{p}(E(\Psi)),B_{1}^{(p)}(\Psi),$ for
\[\int_{X}^{\oplus_{p}}l^{p}(E(\Phi)_{x})\,d\mu(x),\]
\[\int_{X}^{\oplus_{p}}B_{1}^{(p)}(\Phi_{x})\,d\mu(x).\]

	Then
\[\chi_{A}L^{p}(E(\Psi))=L^{p}(E(\Psi\big|_{A})),\]
\[\chi_{A}B_{1}^{(p)}(\Psi)=B^{(p)}_{1}(\Psi),\]
also it is easy to see that
\[B_{1}^{(p)}(\Phi)=B_{1}^{(p)}(\Psi).\]

	Considering the exact sequence
\[\begin{CD} 0@>>> L^{p}(E(\Phi\setminus \Psi) @>>> \frac{L^{p}(E(\Phi))}{B^{(p)}_{1}(\Psi)} @>>> \frac{L^{p}(E(\Psi))}{B^{(p)}_{1}(\Psi)} @>>> 0,\end{CD}\]
we have
\begin{align*}
c^{(p)}_{1,\Sigma}(\R)&\leq c(\Phi\setminus \Psi)+\dim_{\Sigma,l^{p}}(L^{p}(E(\Phi))/B^{(p)}_{1}(\Psi),\R)\\
&=1-\mu(A)+\dim_{\Sigma,l^{p}}(L^{p}(E(\Phi))/B^{(p)}_{1}(\Psi),\R).
\end{align*}

	By Proposition \ref{P:compress}, we thus have
\[c^{(p)}_{1,\Sigma}(\R)\leq 1-\mu(A)+\mu(A)c^{(p)}_{1,\Sigma}(\R_{A}).\]

	Rearraging proves the inequality.

\end{proof}

\begin{cor} Let $(X,\mu,\R)$ be a sofic, ergodic, finitely presented, discrete, measure-preserving equivalence relation. If for some $p$ we have 
\[\inf_{\Sigma}c^{(p)}_{1,\Sigma}(\R)>1,\]
where the infimum is over all sofic approximations, then the fundamental group of $\R$ is trivial.

\end{cor}

	We will deduce more about $c^{(p)}_{1,\Sigma}(\R)$ in the non-hyperfinite case, but we will first need to discuss the discrete Hodge decomposition for amenable graphs.

 Let $G$ be a countably infinite graph of unifomrly bounded degree. Since $G$ is infinite, $\delta$ is always injective, we say that $G$ is \emph{amenable} if for some $1\leq p<\infty,$ we have $\delta(l^{p}(V(G))$ is a closed subspace of $l^{p}(E(G)).$ Equivalently, there is some $C>0$ so that 
\[\|\delta f\|_{p}\geq C\|f\|_{p}.\]

	Note that if $p$ is as above, then for all $1<q<\infty,$ we have $\delta(l^{q}(G))$ is closed in $l^{q}(E(G)).$ For if $\delta(l^{q}(G))$ were not closed, then we could find $f_{n}\in l^{q}(G)$ of norm one so that $\|\delta f_{n}\|_{q}\to 0,$ by the triangle inequality we have $\|\delta |f_{n}|\|_{q}\to 0,$ and then we find that  $\|\delta |f_{n}|^{q/p}\|_{p}\to 0,$ which contradicts the fact that $\|\delta f\|_{p}\geq C\|f\|_{p}.$ 

	By duality $G$ is amenable if and only if $\partial$ is surjective as an operator from $l^{p}(E(G))\to l^{p}(V(G))$ for some $1<p<\infty,$ and this is also equivalent to saying that $\partial$ is surjective as an operator from $l^{p}(E(G))\to l^{p}(V(G))$ for all $1<p<\infty.$

For notation we let $\Delta=\partial \circ \delta.$

\begin{proposition} Let $G$ be an infinite amenable graph of uniformly bounded degree, then $\Delta$ is invertible as an operator from $l^{p}(V(G))\to l^{p}(V(G))$ for all $1<p<\infty.$ \end{proposition}

\begin{proof} Let $d(x)$ be the degree of $x,$ and let $M_{d}$ be the operator on $l^{p}(V(G))$ given by multiplication by $d.$ Define
\[Af(x)=\frac{1}{d(x)}\sum_{y:[x,y]\in E(G)}f(y),\]
and note that $\Delta=M_{d}(A-\id).$ 

	Regard $d$ as a measure on $V(G),$ then since $G$ has uniformly bounded degree we know that 
\[l^{p}(V(G))=l^{p}(V(G),d)\]
with equivalent norms. Regard $\delta$ as an operator from $l^{p}(V(G),d)\to l^{p}(E(G))$ and let $-\partial_{d}$ be its adjoint, also let $\Delta_{d}=\partial_{d}\circ \delta.$ Since
\[\ip{\delta f,g}_{l^{p}(E(G))}=-\ip{f,\partial g}_{l^{p}(V(G))}=-\ip{f,M_{d}M_{d^{-1}}\partial g}_{l^{p}(V(G))}=\]
\[-\ip{f,M_{d^{-1}}\partial g}_{l^{p}(V(G),d)},\]
we find that $\partial_{d}=M_{d^{-1}}\partial,$ so
\[\Delta_{d}=M_{d^{-1}}\Delta=A-\id,\]
hence it suffices to show that $\Delta_{d}$ is invertible for all $1<p<\infty$ as an operator from $l^{p}(V(G),d)\to l^{p}(V(G),d).$ 

	Let $\varepsilon>0$ be such that $\|\delta f\|_{l^{2}(E(G))}\geq \varepsilon\|f\|_{l^{2}(V(G),d)},$ then
\[\varepsilon^{2}\leq -\Delta_{d},\]
as an operator on $l^{2}(V(G),d).$ Since $-\Delta_{d}=1-A,$ this implies that $ A\leq 1-\varepsilon^{2}$ as an operator on $l^{2}(V(G),d).$ 

	Thus
\[|\ip{Af,f}_{l^{2}(V(G),d)}|\leq \ip{A|f|,|f|}_{l^{2}(V(G),d)}\leq (1-\varepsilon^{2})\|f\|_{2}^{2}.\]

	Since $A$ is a a self-adjoint operator, this implies that $\|A\|_{l^{2}(V(G),d)\to l^{2}(V(G),d)}<1.$ Since $\|A\|_{l^{1}(V(G),d)\to l^{1}(V(G),d)}\leq 1,\|A\|_{l^{\infty}(V(G),d)\to l^{\infty}(V(G),d)}\leq 1,$ by interpolation we find that there is a $C_{p}<1$ so that 
\[\|A\|_{l^{p}(V(G),d)\to l^{p}(V(G),d)}\leq C_{p}.\]

	Thus $\Delta_{d}$ is invertible, as desired.

\end{proof}

\begin{cor}[Discrete Hodge Decomposition] Let $G$ be an infinite non-amenable graph of uniformly bounded degree, then for every $1<p<\infty$ we have the direct sum decomposition
\[l^{p}(E(G))=Z_{1}^{(p)}(G)\oplus B^{1}_{(p)}(G),\]
further a projection onto $B^{1}_{(p)}(G)$ relative to this decomposition may be given by $\delta \circ \Delta^{-1}\circ \partial.$ 
\end{cor}

	To apply this to the case of equivalence relations, we prove the following Lemma.

\begin{lemma} Let $(X,\mu,\R)$ be a finite cost discrete measure-preserving equivalence relation with $\mathcal{O}_{x}$ infinite for almost every $x$.  The following are equivalent

(i) There is a finite subset $\Phi\subseteq [[\R]],$  such that for almost every $x,$ the graph $\Phi_{x}$ is not amenable,

(ii) for every $\R$-invariant measurable $A\subseteq X$ with $\mu(A)>0$ we have $\R_{A}$ is not hyperfinite,

(iii) for every $A\subseteq X$ with $\mu(A)>0$ we have that $\R_{A}$ is not hyperfinite. 
\end{lemma}

\begin{proof} It is clear that (iii) implies (ii).

	The fact that (ii) implies (i) is the content of Lemma 9.5 in \cite{KM}.

	Suppose (iii) fails and (i) holds. Let $A$ with $\mu(A)>0$ be such that $\R_{A}$ is hyperfinite, let $B$ be the $\R$-saturation of $A.$  Since $(i)$ holds, we know that 
\[C_{x}=\inf_{\substack{f\in l^{1}(\mathcal{O}_{x}),\\ \|f\|_{1}=1}}\|\delta_{\Phi_{x}}f\|_{1}>0,\]
for almost every $x\in X$ and is constant of equivalence classes. Thus replacing $A$ with a subset, we may assume that there is  a $C>0$ so that 
\[\|\delta_{\Phi_{x}}f\|_{1}\geq C\|f\|_{1},\]
for all $x\in B.$

	Since $\R_{A}$ is hyperfinite, we may find measurable fields of vectors $\xi^{(n)}_{x}\in l^{1}(\mathcal{O}_{x}\cap A)$ so that $\|\xi^{(n)}_{x}\|_{1}=1,$ and $\|\xi^{(n)}_{x}-\xi^{(n)}_{y}\|_{1}\to 0$ for $(x,y)\in \R_{A}.$ Let $\{\phi_{j}\}_{j\in J}\subseteq [[\R]]$ with $J$ countable be such that $\{\ran(\phi_{j})\}_{j\in J}$ is a disjoint family, $\dom(\phi_{j})\subseteq A,$ and 
\[B=\bigcup_{j\in J}\ran(\phi_{j}).\]

	Define $\lambda^{(n)}_{x}\in \Meas(l^{1}(\mathcal{O}_{x})$ for $x\in X$ by $\lambda^{(n)}_{x}=\xi^{(n)}_{\phi_{j}^{-1}(x)}$ if $x\in B$ and $j$ is such that $x\in \ran(\phi_{j}),$ and $\lambda^{(n)}_{x}=0$ for $x\notin B.$ 

	Define $\zeta^{(n)}\in \Meas(l^{1}(\mathcal{O}_{x}),$ by $\zeta^{(n)}_{x}(y)=\lambda^{(n)}_{y}(x).$ Then
\[\int_{X}\|\delta_{\Phi_{x}}\zeta^{(n)}_{x}\|_{1}\,d\mu(x)\leq \int_{B}\sum_{\phi\in \Phi}\|\lambda^{(n)}_{y}-\lambda^{(n)}_{\phi(y)}\|_{1}\chi_{\dom(\phi)}(y),d\mu(y),\]
and since $\Phi$ has finite cost, this goes to zero by the Dominated Convergence Theorem. But on the other hand,
\[\int_{X}\|\delta_{\Phi_{x}}\zeta^{(n)}_{x}\|_{1}\,d\mu(x)\geq C\int_{X}\|\zeta^{(n)}_{x}\|_{1}\,d\mu(x)=C\int_{\R}|\lambda^{(n)}_{x}(y)|\,d\overline{\mu}(x,y)=C\mu(B),\]
which is a contradiction.

\end{proof}

If $\Phi$ is a graphing of $\R,$ we may define the $l^{p}$-cohomology space of $\R$ as the direct integral of $Z_{1}^{(p)}(\Phi_{x})/B_{1}^{(p)}(\Phi_{x})$ and we denote it by $H_{1}^{(p)}(\Phi).$  We set $\beta^{(p)}_{1,\Sigma}(\phi)=\dim_{\Sigma,l^{p}}(H_{1}^{(p)}(\Phi),\R),$ $\underline{\beta}^{(p)}_{1,\Sigma}(\phi)=\underline{\dim}_{\Sigma,l^{p}}(H_{1}^{(p)}(\Phi),\R).$

\begin{cor} Let $(X,\mu,\R)$ be a discrete, sofic,  measure-preserving equivalence relation such that $\R_{A}$ is not hyperfinite for any $A\subseteq X$ with $\mu(A)>0.$  Suppose $\R$ has finite cost and is finitely presented, and fix a sofic approximation $\Sigma$ of $\R.$ Then for any graphing $\Phi$ of $\R,$ we have
\[c^{(p)}_{1,\Sigma}(\R)\leq \beta^{(p)}_{1,\Sigma}(\Phi)+1,\]
\[\underline{c}^{(p)}_{1,\Sigma}(\R)\leq \underline{\beta}^{(p)}_{1,\Sigma}(\Phi)+1,\]

\end{cor}

\begin{proof} First, express $\Phi=\bigcup_{n}\Phi^{(n)}$ by the above Lemma, we find that up to sets of measure zero,
\[X=\bigcup_{n=1}^{\infty}\{x:\Phi^{(n)}_{x}\mbox{ is not amenable }\},\]
and each of the above sets is $\R$-invariant. From this, it is not hard to see that we may choose $\Phi^{(n)}$ so that for every $n,$ either $\Phi^{(n)}_{x}$ is non-amenable or zero.

 By the discrete Hodge decomposition, we have the following exact sequence
\[\begin{CD} 0@>>> B^{1}_{(p)}(\Phi^{(n)}) @>>> \frac{L^{p}(E(\Phi^{(n)}))}{B^{(p)}_{1}(\Phi^{(n)})} @>>> \frac{Z^{(p)}_{1}(\Phi)}{B^{(p)}_{1}(\Phi^{(n)})} @>>> 0,\end{CD}\]
now apply subadditivity under exact sequences, and Lemma \ref{L:finprescontinuity} to complete the proof.

\end{proof}

\begin{cor} Fix $n\in \NN,$ suppose $\R$ is the equivalence relation induced by a free action of $\FF_{n}$ on a standard probability space $(X,\mu).$ Then for any sofic approximation $\Sigma$ of $\R,$ we have that
\[\underline{c}_{1,\Sigma}^{(p)}(\R)=c_{1,\Sigma}^{(p)}(\R)=n,\]
in particular for $n\geq 1,$
\[\underline{\beta}_{1,\Sigma}^{(p)}(\Phi)\geq n-1\]
for any graphing $\Phi,$ and if $\Phi$ is a treeing of $\R,$ then
\[\underline{\beta}_{1,\Sigma}^{(p)}(\Phi)=\beta_{1,\Sigma}^{(p)}(\Phi)=n-1.\]
 Thus, if $\R$ has infinite orbits and is hyperfinite then
\[\underline{c}_{1,\Sigma}^{(p)}(\R)=c_{1,\Sigma}^{(p)}(\R)=1.\]

\end{cor}

\begin{proof} If $\Phi$ is the graphing provided by the canonical generating set of $\FF_{n},$ then
\[B_{1}^{(p)}(\Phi)=\{0\},\]
\[L^{p}(E(\Phi))\cong L^{p}(\R,\overline{\mu})^{\oplus n},\]
and the proof of the first statement is thus complete.

	By \cite{Me}, Proposition 8.5, we know that $H_{1}^{(p)}(\Phi)$ can be generated by $n-1$ elements, and this proves the upperbound. 

	The last statement follows from the standard fact that a hyperfinite equivalence relation with infinite orbits is induced by a free action of $\ZZ.$

\end{proof}

\begin{proposition} Let $(\R,X,\mu)$ be a discrete measure-preserving equivalence relation such that $\mathcal{O}_{x}$ is infinite for almost every $x\in X.$ Then $\underline{c}_{1,\Sigma}^{(p)}(\R)\geq 1.$\end{proposition}

\begin{proof} By the ergodic decomposition, we can find $\R$-invariant measurable subsets $A,B$ of $X$ so that $\mu(A\cap B)=0,$ with $\R_{A}$ amenable, and $\R_{B}$ has no amenable compression. Let $\alpha\in [\R_{A}]$ generate $\R_{A}.$ Let $\Phi_{0}$ be any countable graphing of $\R_{B},$ and set $\Phi=\{\alpha\}\cup \Phi_{0}.$ Then as $\R$-modules:
\[\frac{L^{p}(E(\Phi))}{B_{1}^{(p)}(\Phi)}=L^{p}(\R_{A},\overline{\mu})\oplus \frac{L^{p}(E(\Phi_{0}))}{B_{1}^{(p)}(\Phi_{0})},\]
and by the Discrete Hodge decomposition we have a surjective $\R$-equivariant map 
\[\frac{L^{p}(E(\Phi_{0}))}{B_{1}^{(p)}(\Phi_{0})}\to \frac{L^{p}(E(\Phi_{0}))}{Z_{1}^{(p)}(\Phi_{0})}\cong L^{p}(\R_{B},\overline{\mu}).\]

	Thus $\frac{L^{p}(E(\Phi))}{B_{1}^{(p)}(\Phi)}$ has an $\R$-equivariant surjection onto $L^{p}(\R,\overline{\mu})$ and this completes the proof.

\end{proof}

$\mathbf{Acknowledgments}.$ The author would like to thank Dimitri Shlyakhtenko, Yehuda Shalom, and Thomas Sinclair  for suggesting to consider Dimension Theory for Equivalence Relations, and Dimitri Shlyakhtenko for his helpful advice on the problem.

\end{document}